\documentclass{article}
\usepackage{hyperref}
\usepackage{graphicx}
\usepackage{indentfirst}
\usepackage{amsmath,amsfonts,amsthm,amssymb,color}
\usepackage{mathrsfs}
\usepackage{amscd}
\usepackage{cd}

\def\co{\colon\thinspace}
\DeclareMathAlphabet{\mathsfsl}{OT1}{cmss}{m}{sl}

\newcommand{\alg}{\mathcal{A}}
\newcommand{\CK}{CKh}
\newtheorem{thm}{Theorem}[section]
\newtheorem{lem}[thm]{Lemma}
\newtheorem{cor}[thm]{Corollary}
\newtheorem{prop}[thm]{Proposition}

\newtheorem{theorem}{Theorem}
\newtheorem{proposition}[theorem]{Proposition}

\theoremstyle{definition}

\newtheorem{rem}[thm]{Remark}
\newcommand{\khaction}{f}

\newcommand{\CKr}{CKh^{{r}}}
\newcommand{\Khr}{Kh^{{r}}}
\newcommand{\modzero}{\F[X_0,...,X_{n-1}]/(X_0^2,...,X_{n-1}^2)}
\newcommand{\modzerofrac}{ \frac{\F[X_0,...,X_{n-1}]}{(X_0^2,...,X_{n-1}^2)}}

\newcommand{\curve}{a}
\newcommand{\Curve}{A}

\newcommand{\modred}{\F[X_1,...,X_{n-1}]/(X_1^2,...,X_{n-1}^2)}

\newcommand{\subcube}{(C(L),d)}
\newcommand{\HFmod}{\Lambda^*(H_1(Y;\mathbb Z)/\mathrm{Tors})}

\newcommand{\etab}{\mbox{\boldmath${\eta}$}}
\newcommand{\HF}{HF}

\def\endproofof{\relax\ifmmode\expandafter\endproofmath\else
  \unskip\nobreak\hfil\penalty50\hskip.75em\hbox{}\nobreak\hfil\bull
  {\parfillskip=0pt \finalhyphendemerits=0 \bigbreak}\fi}
\def\endproofofmath$${\eqno\bull$$\bigbreak}

\def\endproof{\relax\ifmmode\expandafter\endproofmath\else
  \unskip\nobreak\hfil\penalty50\hskip.75em\hbox{}\nobreak\hfil\bull
  {\parfillskip=0pt \finalhyphendemerits=0 \bigbreak}\fi}
\def\endproofmath$${\eqno\bull$$\bigbreak}
\def\bull{\vbox{\hrule\hbox{\vrule\kern3pt\vbox{\kern6pt}\kern3pt\vrule}\hrule}}

\newcommand{\R}{\mathbb{R}}
\newcommand{\T}{\mathbb{T}}

\newcommand{\Z}{\mathbb{Z}}

\newcommand{\cm}{\cdot}

\newcommand\SpinC{\mathrm{Spin}^c}
\newcommand{\F}{\mathbb F}

\newcommand\relspinc{\underline{\spinc}}

\newcommand\Filt{\mathcal F}

\newcommand\x{\mathbf x}

\newcommand\s{\mathbf s}

\newcommand\p{\mathbf p}
\newcommand\q{\mathbf q}
\newcommand\y{\mathbf y}

\newcommand\ModSphere{\ModFlow\left({\mathbb S}\longrightarrow 
\Sym^{g-1}(\Sigma_{1})\times \Sym^2(\Sigma_{2})\right)}
\newcommand\ModSpheres\ModSphere
\newcommand\CF{CF}

\newcommand\CFa{\widehat{CF}}
\newcommand\CFp{\CFb}
\newcommand\CFm{\CF^-}

\newcommand\CFinf{CF^\infty}

\newcommand\CFb{CF^+}
\newcommand\HFa{\widehat{HF}}

\newcommand\UnparModSp{\widehat \ModSp}
\newcommand\UnparModFlow\UnparModSp
\newcommand\Mod\ModSp

\newcommand{\spinc}{\mathfrak s}

\newcommand{\spinct}{\mathfrak t}

\newcommand\ModMaps{\mathcal M}
\newcommand\ModSp\ModMaps

\newcommand\Ta{{\mathbb T}_{\alpha}}
\newcommand\Tb{{\mathbb T}_{\beta}}

\newcommand\spincrel\relspinc

\newcommand\HFc{\HF^\circ}

\newcommand\Dual{\mathcal D}
\newcommand\Duality\Dual

\newcommand\ons{Ozsv{\'a}th and Szab{\'o}}
\newcommand\os{{Ozsv{\'a}th--Szab{\'o}}}

\begin{document}

\title{Khovanov module and the detection of unlinks}

\author{{\large Matthew HEDDEN}\\{\normalsize Department of Mathematics, Michigan State University}
\\{\normalsize A338 WH, East Lansing, MI 48824}\\{\small\it Emai\/l\/:\quad\rm mhedden@math.msu.edu}
\\ \\{\large Yi NI}\\{\normalsize Department of Mathematics, Caltech, MC 253-37}\\
{\normalsize 1200 E California Blvd, Pasadena, CA
91125}\\{\small\it Emai\/l\/:\quad\rm yini@caltech.edu}}

\date{}
\maketitle

\begin{abstract}
We study a module structure on Khovanov homology, which we show is natural under the Ozsv\'ath--Szab\'o spectral
sequence to the Floer homology of the branched double cover. As an application, we show that this module structure
detects trivial links.  A key ingredient of our proof is that the
$\Lambda^*H_1$--module structure on Heegaard Floer homology
detects $S^1\times S^2$ connected summands.
\end{abstract}

\section{Introduction}
The  Jones polynomial \cite{Jones} has had a tremendous impact since its discovery, leading to an array of invariants  of knots and $3$--manifolds.    The meaning of these invariants is rather elusive.  In fact it remains unknown whether there exists a non-trivial knot with the same Jones polynomial as the unknot.  

Khovanov discovered a refinement of the Jones polynomial which assigns bigraded homology groups to a link \cite{Kh}.  The Jones polynomial is recovered by taking the Euler characteristic of Khovanov's homology and keeping track of the additional grading by the exponent of a formal variable $q$.    One could hope that the geometry contained in this refinement is more transparent, and a step towards understanding the question above would be to determine whether Khovanov homology detects the unknot.  

The pure combinatorial nature of Khovanov homology, however, makes direct inspection of its ability to detect the unknot quite difficult.      Surprisingly, the most fruitful approach to such questions has been through connections with the world of gauge theory and Floer homology.  Indeed, a recent result of Kronheimer and Mrowka uses an invariant of knots arising from instanton Floer homology to prove that  Khovanov homology detects the unknot.  More precisely,  \cite[Theorem 1.1]{KM2010} states that 
$$\mathrm{rk}\ \Khr(K)=1  \Longleftrightarrow K \mathrm{\ is\ the\ unknot.}$$
Their result had numerous antecedents \cite{Baldwin2010,BaldwinLevine,BaldPlam, GW1, Hedden, HN, TangleUnknotting, Roberts2007}, most aimed at generalizing or exploiting a spectral sequence discovered by \ons \ \cite{OSzBrCov} which begins at Khovanov homology and converges to the Heegaard Floer homology of the branched double cover.  

Kronheimer and Mrowka's theorem raises the natural question of whether Khovanov homology also detects an unlink of numerous components.  Interest in this question is heightened by the existence of infinite families of links which have the same Jones polynomial as unlinks with two or more components \cite{EKT,Thistlethwaite}.  One immediately observes that the rank of Khovanov homology  does not detect unlinks.  This is demonstrated  by the Hopf link, which has the same rank   Khovanov homology as the $2$--component unlink (see Subsection \ref{subsec:hopf} for a discussion).    Thus more information   must be brought to the table if one hopes to use Khovanov homology to detect unlinks. Using a small portion of the bigrading on Khovanov homology we were able to make initial progress on the question of unlink detection in \cite{HN}.  There we   showed  \cite[Corollary~1.4]{HN}
$$ Kh(L) \text{\ detects the unknot }\Longleftrightarrow\newline Kh(L) \text{\ detects the two-component unlink}$$
\noindent which, together with Kronheimer and Mrowka's theorem settles the question for unlinks of two components.   Unfortunately, the strategy in \cite{HN} could not be extended.  

The purpose of this article is to show that Khovanov homology detects unlinks by exploiting a module structure inherent in the theory.  In the next section we define a module structure on the Khovanov chain complex, and prove that the induced module structure on the Khovanov homology groups with $\F=\Z/2\Z$ coefficients is an invariant of the link.  More precisely, we have 
\begin{proposition} \label{thm:module}  Let $L\subset S^3$ be a link of $n$ components.  The  Khovanov homology $Kh(L;\F)$ is a module over the ring
 $$\F[X_0,...,X_{n-1}]/(X_0^2,..., X_{n-1}^2).$$ The isomorphism type of this module is an invariant of $L$.
 \end{proposition}  The action of $X_i$ is defined analogously to reduced Khovanov homology \cite{KhPatt}, and is obtained from the chain map  on $CKh(K)$ induced by the link cobordism which merges an unknotted circle to the $i$-th component of $L$.    Our main theorem shows that the Khovanov module detects unlinks:

\begin{theorem}\label{thm:main}  Let $L$ be a link of $n$ components.  If there is an isomorphism of modules  $$Kh(L;\F)\cong \F[X_0,...,X_{n-1}]/(X_0^2,..., X_{n-1}^2),$$ then $L$ is the unlink.
\end{theorem}

Theorem~\ref{thm:main}  will be proved in the context of the spectral sequence from Khovanov homology to the Heegaard Floer homology of the branched double cover of $L$.  This latter invariant also has a module structure.  Indeed $\widehat{HF}(Y)$ is   a module over the exterior algebra on $H^1(Y;\F)$, the first singular cohomology of $Y$ with coefficients in $\F$.    A key ingredient in our proof, Theorem \ref{thm:Khnaturality}, is to refine  \os's spectral sequence to incorporate both module structures.    A consequence of this result is the following theorem, which indicates the flavor of our refinement:

\begin{theorem}\label{sskhmodule}  There is a spectral sequence of modules, starting at the reduced Khovanov module of the mirror of a link and converging to the Floer homology of its branched double cover.  \end{theorem}

   Armed with this structure, we show that if the Khovanov module is isomorphic to that of the unlink, then the  Floer homology of the branched double cover is  isomorphic to  $\F[X_1,...,X_{n-1}]/(X_1^2,..., X_{n-1}^2)$ as a module (see Proposition \ref{collapse}).  The second main ingredient in our proof is the following theorem, which says that Floer homology detects $S^1\times S^2$ summands in the prime decomposition of a $3$--manifold.

\begin{theorem}\label{thm:HFmodUnique}
Suppose that $\widehat{HF}(Y;\F)\cong \F[X_1,...,X_{n-1}]/(X_1^2,..., X_{n-1}^2)$ as a module.  Then $Y\cong M\#(\#^{n-1}(S^1\times S^2))$, where $M$ is an integer homology sphere satisfying $\HFa(M)\cong \F$.
\end{theorem}

\noindent This theorem seems  interesting in its own right, and complements an array of results on the faithfulness of the Floer invariants for particularly simple manifolds. \

 The proof of the main theorem  follows quickly from Theorems \ref{sskhmodule} and \ref{thm:HFmodUnique}.   Indeed, if the Khovanov module of  $L$ is isomorphic to that of the unlink, then the two results imply that the  branched double cover of $L$ is homeomorphic to the connected sum of $S^1\times S^2$'s with an integer homology sphere whose Floer homology has rank one.  Using classical tools from equivariant topology, we then see that $L$ is a split link, each component of which has the Khovanov homology of the unknot.  Kronheimer and Mrowka's theorem then tells us that each component is unknotted.

\vskip0.1in
\noindent {\bf Outline:} This paper is organized as follows. In Section~2 we will review Khovanov homology and its module structure.  There we prove that the module structure is an invariant. In Section~3 we will give the necessary background on Heegaard Floer homology, especially the module structure and an $A_{\infty}$ type relation. In Section~4 we relate the Heegaard Floer module structure with \ons's link surgeries spectral sequence. As an application, we connect the module structure on Khovanov homology of a link with the module structure on the Heegaard Floer homology of the branched double cover of the link. Section~5 is devoted to a nontriviality theorem for the module structure on Heegaard Floer homology. The proof is similar to the standard nontriviality theorem in Heegaard Floer theory. The main theorem is proved in Section~6, using the results in the previous two sections as well as a homological algebra argument.

\

\noindent{\bf Acknowledgements.}\quad This work was initiated
when the authors participated the ``Homology Theories of Knots and
Links'' program at MSRI, and was carried out further when the authors visited Simons Center for Geometry and Physics. We are grateful to Ciprian Manolescu and
Tomasz Mrowka for helpful conversations. We also wish to thank Robert Lipshitz, Sucharit Sarkar and the referee for pointing out a mistake in the proof of Proposition~\ref{prop:KhInv}. Special thanks are due to Sucharit Sarkar for suggesting a way to fix the mistake. 
The first author was partially supported by NSF grant numbers DMS-0906258 and DMS-1150872 and an Alfred P. Sloan Research Fellowship. The second author was
partially supported by an AIM Five-Year Fellowship, NSF grant
numbers DMS-1021956, DMS-1103976, and an Alfred P. Sloan Research Fellowship.  

\section{Preliminaries on Khovanov homology}
\label{sec:Kh}

Khovanov homology is a combinatorial invariant of links in the $3$--sphere which refines the Jones polynomial.  In this section we briefly recall some background on this invariant, but will assume familiarity with \cite{Kh}.   Our primary purpose is to establish notation and define a module structure on the Khovanov chain complex (or homology) which is implicit in \cite{Kh,KhPatt}, but which has not attracted much attention.    

To a diagram $D$ of a link $L\subset S^3$, Khovanov associates a bigraded cochain complex, $(\CK_{i,j}(D),\partial)$ \cite{Kh}.   The $i$--grading  is the cohomological grading, in the sense that is raised by one by the coboundary operator $$\partial\co \CK_{i,j}(D)\rightarrow \CK_{i+1,j}(D).$$   The $j$--grading is the so-called ``quantum" grading, and is preserved by the differential.  One can think of this object as a collection of complexes in the traditional sense, with the collection indexed by an additional grading.  The homology of these complexes does not depend on the particular diagram chosen for $L$, and produces an invariant $$Kh(L):=\bigoplus_{i,j} Kh_{i,j}(L)$$ called the {\em Khovanov (co)homology of L}.  Taking the Euler characteristic  in each quantum grading, and keeping track of this with a variable $q$, we naturally obtain a Laurent polynomial 
$$V_K(q)=\sum_j {\big(}\sum_i(-1)^{i}\mathrm{rank}\ Kh_{i,j}(L){\big)}\cdot q^j.$$
This polynomial agrees with the (properly normalized) Jones polynomial.  

The complex $\CK(D)$ is obtained by applying a $(1+1)$--dimensional TQFT to the hypercube of complete resolutions of the diagram, and the algebra assigned to a single unknotted circle by this structure is $\alg=\F[X]/(X^2)$. The product on this algebra is denoted $$m\co \mathcal A\otimes\mathcal A\to\mathcal A.$$ There is also a coproduct $$\Delta\co\mathcal A\to\mathcal A\otimes\mathcal A$$ which is defined by letting
$$\Delta(\mathbf 1)=\mathbf 1\otimes X+X\otimes\mathbf 1,\qquad\Delta(X)=X\otimes X.$$
  Our purposes will not require strict knowledge of gradings, and for convenience we relax Khovanov homology to  a relatively $\Z\oplus\Z$--graded theory.  This means that we consider Khovanov homology up to overall shifts in either the homological or quantum grading.    In these terms,  the quantum grading of $\mathbf 1\in\alg$ is two greater than that of $X$, and the homological grading is given by the number of crossings resolved with a $1$--resolution, in a given complete resolution.

Section 3 of \cite{KhPatt} describes a module structure on the Khovanov homology of a knot which we now recall.  Given a diagram $D$ for a knot $K$, let $p\in K$ be a marked point. Now place an unknotted circle next to $p$ and consider the saddle cobordism that merges the circle with a segment of $D$ neighboring $p$.   Cobordisms induce maps between Khovanov complexes and, as such, we have a map $$\alg\otimes \CK(D)\overset{\khaction_p}\longrightarrow \CK(D).$$  

\begin{figure}
\centering
\def\svgwidth{2.5in}
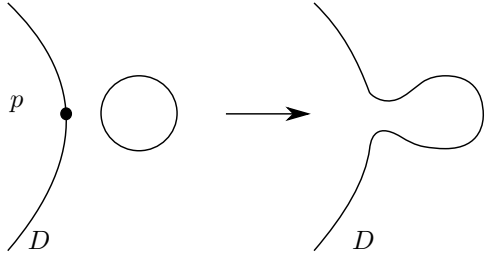
\caption{\label{fig:merge}  Merging an unknotted circle with $D$.}

\end{figure}

\noindent Explicitly, the map is  the algebra multiplication $\alg\otimes \alg\rightarrow \alg$ between vectors associated to the additional unknot and the unknot in each complete resolution of the diagram which contains the marked point. We denote the induced map on homology by $F_p$.  It is independent of the pair $(D,p)$, in a sense made precise by the following proposition.

\begin{prop} {\em(\cite{KhPatt})} \label{prop:red} Let $D$ and $D'$ be diagrams for a knot $K$, and let $p\in D$, $p'\in D'$ be marked points on each diagram.  Then there is a commutative diagram:
$$
\begin{CD}
 \alg\otimes \CK(D) @>\khaction_{p}>>  \CK(D)\\
@VV1\otimes eV @VVeV\\
 \alg\otimes \CK(D') @>\khaction_{p'}>>  \CK(D'),
\end{CD} 
$$
where $e$ is a chain homotopy equivalence.
\end{prop}
As observed in \cite{KhPatt}, the proof is quite simple provided one knows the proof of invariance for Khovanov homology.  According to \cite{Kh}, a sequence of Reidemeister moves from $D$ to $D'$ induces a chain homotopy equivalence $e$ between the associated complexes.  If these moves occur in the complement of a neighborhood of $p$, then $e$ obviously commutes with the map $\khaction_p$, where $p$ is regarded as a point in both $D$ and $D'$.  Now it suffices to observe that Reidemeister moves which cross under or over $p$ can be traded for a sequence of moves which do not (simply drag the segment of knot in question the opposite direction over the plane). 

 Thus  a choice of marked point endows the Khovanov homology of a knot with the structure of an $\alg$--module.    One immediate application of this module structure is that it allows for the definition of the reduced Khovanov homology. Indeed, we can consider  $\F=X\cdot \alg$ as an $\alg$--module, and correspondingly define the {\em reduced Khovanov complex} \footnote{This is not quite Khovanov's definition, but is isomorphic to it.}
 $$\CKr(D):= \CK(D)\otimes_\alg \F.$$
 
The structure we use is a straightforward generalization of this  construction to links. Consider a diagram $D$ for a link $L\subset S^3$ of $n$ components.  For each component $K_i\subset L$ pick a marked point $p_i\subset K_i$ and obtain a map:
$$ \begin{CD} \alg\otimes CKh(D) @>{\khaction_{p_i}}>>CKh(D).\end{CD} $$

\noindent Now for each $i$, let $x_i(-):=f_{p_i}(X \otimes -)$ denote the chain map induced from this module structure, and $$ X_i\co Kh(L) {\longrightarrow} Kh(L), \ \ i=0,...,n-1,$$
denote the map induced on homology. From the definition we see that $x_i$, and hence $X_i$, commute and satisfy $x_i^2=0.$ Thus we can regard Khovanov homology as a module over $\modzero$.   Showing that  this module structure is a link invariant is not as straightforward as above.  This is due to the fact that we may not be able to connect two multi-pointed link diagrams by a sequence of Reidemeister moves which occur in the complement of all the basepoints.  One may expect, then, that the module structure is an invariant only of the pointed link isotopy class.  In fact the isomorphism type of the  module structure on  $Kh(L;\F)$ does not depend on the choice of marked points or link diagram, as the following proposition shows.

\begin{prop} \label{prop:KhInv}
Let $D$ and $D'$ be diagrams for a link $L$ of $n$ components and let $p_i\in D_i$, $p_i'\in D_i'$ be marked points, one chosen on each component of the diagram.  Then there is a commutative diagram:
$$
\begin{CD}
\displaystyle \modzerofrac\otimes Kh(D) @>\khaction_{\mathbf{p}_*}>>  Kh(D)\\
@VV1\otimes e_*V @VVe_*V\\
\displaystyle \modzerofrac\otimes Kh(D') @>\khaction_{\mathbf{p}'_*}>>  Kh(D'),
\end{CD} 
$$
where $e_*$ is an isomorphism induced by a chain homotopy equivalence.
\end{prop}

\noindent The proof of Proposition~\ref{prop:KhInv} uses an argument suggested by Sucharit Sarkar.

Any two diagrams $D,D'$ are related by a sequence of Reidemeister moves. If a Reidemeister move occurs in the complement of neighborhoods of all marked points $p_i$, then the commutative diagram obviously exists. Thus we only need to consider the case of isotoping an arc over or under a marked point or, equivalently, of sliding a basepoint past a crossing. This case is contained in the following lemma.

\begin{lem}\label{lem:xp=xq}
Suppose that $D$ is a link diagram, and $\chi_0$ is a crossing of $D$. Suppose that $p,q$ are two marked points on a strand passing $\chi_0$, such that $p,q$ are separated by $\chi_0$. Let $x_p,x_q\co CKh(D)\to CKh(D)$ be the module multiplications induced by the marked points $p,q$. Then $x_p$ and $x_q$ are homotopy equivalent.
\end{lem}

Consider the cube of resolutions of $D$. If $IJ$ is an oriented edge of the cube, let $\partial_{IJ}\co C(I)\to C(J)$ be the map induced by the elementary cobordism corresponding to the edge.
Let $I_0,I_1$ be any two complete resolutions of $D$ that differ only at the crossing $\chi_0$, where $I_0$ is locally the $0$--resolution and $I_1$ is locally the $1$--resolution. Suppose that the immediate successors of $I_0$ besides $I_1$ are $J_0^1,J_0^2,\dots,J_0^r$. Then the immediate successors of $I_1$ are $J_1^1,J_1^2,\dots,J_1^r$, where  $J_0^i$ and $J_1^i$ differ only at the crossing $\chi_0$.

We write $CKh(D)$ as the direct sum of two subgroups
$$CKh(D)=CKh(D_0)\oplus CKh(D_1).$$
Here $D_j$ is the $j$--resolution of $D$ at $\chi_0$ for $j=0,1$. Then $CKh(D_1)$ is a subcomplex of $CKh(D)$, and $CKh(D_0)$ is a quotient complex of $CKh(D)$.

We define $H\co CKh(D)\to CKh(D)$ as follows. On $CKh(D_0)$, $H=0$; on $CKh(D_1)$, $H\co CKh(D_1)\to CKh(D_0)$ is defined  by the map associated to the elementary cobordism from $D_1$ to $D_0$.
We claim that
\begin{equation}\label{eq:Homotopy}
x_p-x_q=\partial H+H\partial.
\end{equation}
The claim obviously implies Lemma \ref{lem:xp=xq}, and hence Proposition \ref{prop:KhInv}.   To prove the claim, we begin with another lemma.

\begin{lem}\label{lem:X-X}Let $I_0,I_1$ be as above. Then for any $\alpha\in C(I_0)$ and $\beta\in C(I_1)$, the following identities hold:
\begin{eqnarray*}
(x_p-x_q)(\alpha)&=&H\partial_{I_0I_1}(\alpha),\\
(x_p-x_q)(\beta)&=&\partial_{I_0I_1}H(\beta).
\end{eqnarray*}
\end{lem}
\begin{proof}
It is easy to check the following identities in $\mathcal A$:
\begin{equation*}
m\Delta=0,\quad \Delta m(a\otimes b)=Xa\otimes b+a\otimes Xb.
\end{equation*}
Our conclusion then immediately follows.
\end{proof}

By the definition of $H$, we have $H(\alpha)=0$. Moreover, 
\begin{eqnarray*}
H\partial(\alpha)&=&H\big(\partial_{I_0I_1}\alpha+\sum_{i=1}^r\partial_{I_0J_0^i}\alpha\big)\\
&=&H\partial_{I_0I_1}(\alpha),
\end{eqnarray*}
which is equal to $(x_p-x_q)(\alpha)$ by Lemma~\ref{lem:X-X}.  Thus \eqref{eq:Homotopy} holds for $\alpha\in C(I_0)$.

Turning to $\beta\in C(I_1)$, we have\begin{equation}
\partial H(\beta)=\partial_{I_0I_1}H(\beta)+\sum_{i=1}^r\partial_{I_0J_0^i}H(\beta),\label{eq:dH}
\end{equation}
and 
\begin{equation}\label{eq:Hd}
H\partial (\beta)=H\big(\sum_{i=1}^r\partial_{I_1J_1^i}(\beta)\big).
\end{equation}

Let $\overline D$ be the diagram (of a possibly different link) which is obtained from $D$ by changing the crossing $\chi_0$. Then $H$ is the summand of the differential in $CKh(\overline D)$ which is induced by the edges parallel to $I_1I_0$.
It is clear that 
$$
\sum_{i=1}^r\partial_{I_0J_0^i}H(\beta)+H\big(\sum_{i=1}^r\partial_{I_1J_1^i}(\beta)\big)
$$
is equal to $\partial^2\beta$ in $CKh(\overline D)$, which is zero.
So from (\ref{eq:dH}) (\ref{eq:Hd}) and  Lemma~\ref{lem:X-X} we have 
$$(\partial H+H\partial)(\beta)=\partial_{I_0I_1}H(\beta)=(x_p-x_q)(\beta).$$
This finishes the proof of (\ref{eq:Homotopy}), hence Lemma~\ref{lem:xp=xq} follows.

\begin{rem}  Note that we make no claims about the naturality of the module structure. The only result we need is that isotopic links have isomorphic Khovanov modules.  Indeed, it seems likely that the module structure on Khovanov homology is functorial, but only in the category of pointed links; that is, links with a choice of basepoint on each component.
\end{rem}

When considering the relationship with Heegaard Floer homology, it will be more natural to work in the context of reduced Khovanov homology.  Note that reduced Khovanov homology can be interpreted as the homology of the kernel complex $H_*(\mathrm{ker}\{ CKh(L)\overset{x_i}\rightarrow CKh(L)\})$ or, equivalently, as the kernel of the map on homology, $\mathrm{ker}{X_i}$ \cite{Shumakovitch2004}. In particular, the module structure from the proposition descends to an $\modred$--module structure on the reduced Khovanov homology.   Henceforth, we will let $p_0$ be the point chosen to define reduced Khovanov homology, so that $\Khr(L)$ is a module over $\modred$. 

Although the reduced Khovanov homology groups do not depend on the choice of  component $L_0$ containing $p_0$, their module structure will, in general, depend on this choice. Thus we  use $\Khr(L,L_0)$ to emphasize the dependence of the module structure on the component $L_0$ and abuse the notation $\Khr(L)$ when $L_0$ is understood. 

As we shall see, this module structure on the reduced Khovanov homology is  connected to a module structure on the Heegaard Floer homology of the branched double cover through a refinement of the \os \ spectral sequence.

We also note that, since $X_i^2=0$, the module structure equips the Khovanov homology (or reduced Khovanov homology) with the structure of a chain complex.  Thus it makes sense to talk about the {\em homology of Khovanov homology with respect to $X_i$}, which we frequently denote  $H_*(Kh(L),X_i)$.  It follows from \cite{Shumakovitch2004} that $H_*(Kh(L),X_i)=0$ for any $i$.  It is far from true, however,  that $H_*(\Khr(L),X_i)=0$ for each $i$, in general.  

\subsection{Example: The unlink versus the Hopf link}\label{subsec:hopf}
It is simple yet instructive to consider the distinction between the Khovanov module of the two component unlink and the Hopf link.  The former is represented by Khovanov chain complex $CKh(\text{Unlink})=\alg\langle X_0\rangle \otimes_\F\alg\langle X_1\rangle$ with $\partial\equiv 0$.  The homology is thus isomorphic to $\F[X_0,X_1]/(X_0^2,X_1^2)$ as a module, and is supported in a single homological grading.  The reduced Khovanov homology is  $\Khr(\mathrm{Unlink})\cong\F[X]/X^2$.

\begin{figure}
\begin{center}
\begin{picture}(175,95)
\put(0,65){\vector(1,0){70}} \put(65,52){$i$}

\put(35,30){\vector(0,1){70}} \put(39,95){$j$}

\put(15,65){\line(0,1){2}}

\put(25,65){\line(0,1){2}}

\put(45,65){\line(0,1){2}}

\put(55,65){\line(0,1){2}}

\put(35,45){\line(1,0){2}}

\put(35,55){\line(1,0){2}}

\put(35,75){\line(1,0){2}}

\put(35,85){\line(1,0){2}}

\put(33,65){\circle*{4}}
\put(37,65){\circle*{4}}

\put(35,45){\circle*{4}}

\put(35,85){\circle*{4}}

\put(100,65){\vector(1,0){70}} \put(165,52){$i$}

\put(140,0){\vector(0,1){90}} \put(144,85){$j$}

\put(130,65){\line(0,1){2}}

\put(150,65){\line(0,1){2}}

\put(160,65){\line(0,1){2}}

\put(120,65){\line(0,1){2}}

\put(140,45){\line(1,0){2}}

\put(140,55){\line(1,0){2}}

\put(140,75){\line(1,0){2}}

\put(140,35){\line(1,0){2}}

\put(140,25){\line(1,0){2}}

\put(140,15){\line(1,0){2}}

\put(140,5){\line(1,0){2}}

\put(140,65){\circle*{4}}

\put(140,45){\circle*{4}}

\put(120,25){\circle*{4}}

\put(120,5){\circle*{4}}

\end{picture}

\caption{\label{fig:UnlinkHopf}  Khovanov homology of the two-component unlink and the Hopf link, where a black dot stands for a copy of $\mathbb Z$, the $i$--coordinate is the homological grading and the $j$--coordinate is the quantum grading.}
\end{center}
\end{figure}
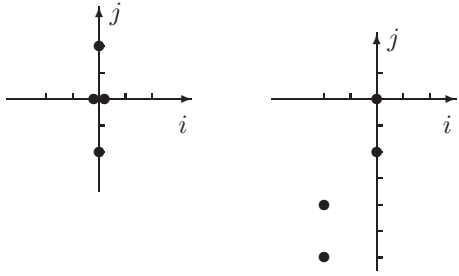

\begin{figure}
\begin{center}
\begin{picture}(265,95)

\put(13,65){\circle*{4}}
\put(17,65){\circle*{4}}

\put(15,45){\circle*{4}}

\put(15,85){\circle*{4}}

\put(13.5,83){\vector(0,-1){16}}
\put(16.5,63){\vector(0,-1){16}}

\put(-5,-7){$X_0$ action}

\put(73,65){\circle*{4}}
\put(77,65){\circle*{4}}

\put(75,45){\circle*{4}}

\put(75,85){\circle*{4}}

\put(76.5,83){\vector(0,-1){16}}
\put(73.5,63){\vector(0,-1){16}}

\put(55,-7){$X_1$ action}

\put(180,65){\circle*{4}}

\put(180,45){\circle*{4}}

\put(160,25){\circle*{4}}

\put(160,5){\circle*{4}}

\put(158.5,23){\vector(0,-1){16}}
\put(178.5,63){\vector(0,-1){16}}

\put(150,-7){$X_0$ action}

\put(240,25){\circle*{4}}

\put(240,5){\circle*{4}}

\put(260,65){\circle*{4}}

\put(260,45){\circle*{4}}

\put(241.5,23){\vector(0,-1){16}}
\put(261.5,63){\vector(0,-1){16}}

\put(230,-7){$X_1$ action}

\end{picture}

\caption{\label{fig:KhMod} Module structure on the Khovanov homology of the two-component unlink and the Hopf link. The $X_0$ and $X_1$ actions are different for the two-component unlink, but are equal for the Hopf link.}
\end{center}
\end{figure}
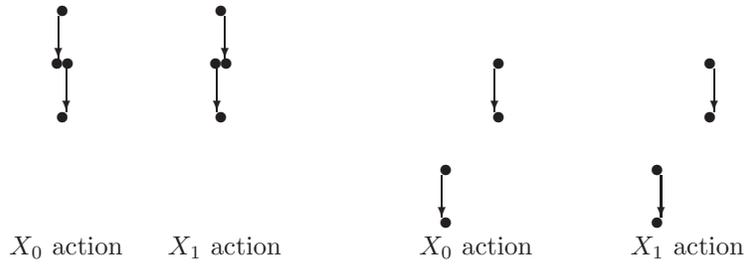

The Hopf link, on the other hand, has Khovanov homology  $$Kh(\mathrm{Hopf})\cong \F\langle a\rangle \oplus\F\langle b\rangle \oplus\F\langle c\rangle \oplus \F\langle d\rangle,$$ with the relative homological gradings of $c$ and $d$  equal, and two greater than those of $a$ and $b$.  The relative quantum grading of $a$ is $2$ lower than that of $b$ which is two lower than $c$ which, in turn, is two lower than $d$.  The module structure is given by 
$$  X_0(b)=X_1(b)=a\ \ \ \ \ X_0(d)=X_1(d)=c. $$
More succinctly, the Khovanov module of the Hopf link is isomorphic to $$\frac{\F[X]}{X^2}\oplus \frac{\F[X]\{2,4\}}{X^2},$$ where we use notation $\{2,4\}$ to denote a shift of $2$ and $4$ in the homological and quantum gradings respectively.  In this module each $X_i$ acts as $X$. 
It follows that the module structure on reduced Khovanov homology is trivial; that is $$\Khr(\mathrm{Hopf})\cong \F\langle a\rangle \oplus \F\langle c\rangle,$$ where $X\in \F[X]/X^2$ acts as zero on both summands.  Thus, the homology of the reduced Khovanov homology with respect to $X$ is $$H_*(\Khr(\mathrm{Hopf}),X)\cong \F\langle a\rangle \oplus \F\langle c\rangle,$$ while for the unlink it is trivial
$$ H_*(\Khr(\mathrm{Unlink}),X)=0.$$

\section{Preliminaries on Heegaard Floer homology}

In this section, we recall the basic theory of Heegaard
Floer homology, with emphasis on the action of
$\Lambda^*(H_1(Y;\mathbb Z)/\mathrm{Tors})$ and twisted
coefficients.  For a  detailed account of the theory, we refer the reader to \cite{OSzAnn1} (see also \cite{OSzSurvey,FloerClay2,OSzICM} for  gentler introductions).

Suppose $Y$ is a closed oriented $3$--manifold, together with a $\SpinC$ structure $\mathfrak
s\in\mathrm{Spin}^c(Y)$. Let
$$(\Sigma,\mbox{\boldmath${\alpha}$},
\mbox{\boldmath$\beta$},z)$$ be an admissible pointed Heegaard diagram for
$(Y,\mathfrak s)$, in the sense of \cite{OSzAnn1}.  To such a diagram we can associate the \os \  infinity chain complex, ${CF}^{\infty}(Y,\mathfrak s)$.  This chain complex is freely generated over the ring $\F[U,U^{-1}]$ by intersection points $\x\in \Ta\cap \Tb$, where $\Ta$ (resp. $\Tb$) is the Lagrangian torus in Sym$^g(\Sigma)$, the $g$--fold symmetric product of $\Sigma$.  The boundary operator is defined by
$$\partial^\infty \x = \sum_{\mathbf y\in\mathbb T_{\alpha}\cap\mathbb T_{\beta}}
\sum_{\{\phi\in\pi_2(\mathbf x,\mathbf
y)|\mu(\phi)=1\}}\#\widehat{\mathcal M}(\phi)\cm U^{n_z(\phi)}\y,$$
where $\pi_2(\x,\y)$ is the set of homotopy classes of Whitney disks connecting $\x$ to $\y$, $\mu$ is the Maslov index, $\#\widehat{\mathcal M}(\phi)$ denotes the  reduction modulo two of the number of unparametrized pseudo-holomorphic maps in the homotopy class $\phi$, and $n_z(\phi)$ denotes the algebraic intersection number of such a map with the holomorphic hypersurface in Sym$^g(\Sigma)$ consisting of unordered $g$--tuples of points in $\Sigma$ at least one of which is the basepoint $z$.

\subsection{The $\Lambda^*(H_1(Y;\mathbb Z)/\mathrm{Tors})$ action} \label{subsec:H_1}

In this subsection we describe an action of $\Lambda^*(H_1(Y;\mathbb Z)/\mathrm{Tors})$ on the Floer homology of a $3$--manifold $Y$. Let $\zeta\subset \Sigma$ be a closed oriented (possibly
disconnected) immersed curve which is in general position with the
$\alpha$-- and $\beta$--curves. Namely, $\zeta$ is transverse to
these curves, and $\zeta$ does not contain any intersection point
of $\alpha$-- and $\beta$--curves.  Note that any closed curve $\zeta_0\subset Y$ can be homotoped in $Y$ to be an immersed curve in $\Sigma$.  If $\phi$ is a topological
Whitney disk connecting $\mathbf x$ to $\mathbf y$, let
$\partial_{\alpha}\phi=(\partial\phi)\cap\mathbb T_{\alpha}$. We
can also regard $\partial_{\alpha}\phi$ as a multi-arc that lies
on $\Sigma$ and connects $\mathbf x$ to $\mathbf y$. Similarly, we
define $\partial_{\beta}\phi$ as a multi-arc connecting $\mathbf
y$ to $\mathbf x$. 

Let
$$\curve^{\zeta}\co {CF}^{\infty}(Y,\mathfrak s)\to {CF}^{\infty}(Y,\mathfrak s)$$ be
defined on generators as

$$\curve^{\zeta}(\mathbf x)=\sum_{\mathbf y\in\mathbb T_{\alpha}\cap\mathbb T_{\beta}}
\sum_{\{\phi\in\pi_2(\mathbf x,\mathbf
y)|\mu(\phi)=1\}}\big(\zeta\cdot(\partial_{\alpha}\phi)\big)\:\#\widehat{\mathcal M}(\phi)\cm U^{n_z(\phi)} \y,$$
where $\zeta\cdot(\partial_{\alpha}\phi)$ is the algebraic
intersection number of $\zeta$ and $\partial_{\alpha}\phi$.  We extend the map to the entire complex by requiring linearity and $U$--equivariance.  As shown in Ozsv\'ath--Szab\'o \cite[Lemma~4.18]{OSzAnn1}, Gromov compactness implies that 
$\curve^{\zeta}$ is a chain map.  Moreover, this chain map clearly respects  the sub-, quotient, and subquotient complexes $\CFm,\CFp,\CFa$, respectively.
Thus $\curve^\zeta$  induces a map, denoted $A^\zeta=(a^\zeta)_*$, on all versions of Heegaard Floer homology.

The following lemma shows that $A^\zeta$ only depends on the homology class $[\zeta]\in
H_1(Y;\mathbb Z)/\mathrm{Tors}$. 

\begin{lem}\label{homologyaction}
Suppose $\zeta_1,\zeta_2\subset \Sigma$ are two curves which are
homologous in $H_1(Y;\mathbb Z)/\mathrm{Tors}$, then $\curve^{\zeta_1}$
is chain homotopic to $\curve^{\zeta_2}$.
\end{lem}
\begin{proof}
Since $\zeta_1$ and $\zeta_2$ are homologous in $H_1(Y;\mathbb Z
)/\mathrm{Tors}$, there exists a nonzero integer $m$ such that
$m[\zeta_1]=m[\zeta_2]\in H_1(Y;\mathbb Z)$. Using the fact that
$$H_1(Y)\cong H_1(\Sigma)/([\alpha_1]\dots,[\alpha_g],[\beta_1],\dots,[\beta_g]),$$
we conclude that there is a $2$--chain $B$ in $\Sigma$, such that
$\partial B$ consists of $m\zeta_2$, $m(-\zeta_1)$ and copies of
$\alpha$--curves and $\beta$--curves. Perturbing $B$ slightly, we
get a $2$--chain $B'$ such that
$$\partial B'=m\zeta_2-m\zeta_1+\sum (a_i\alpha'_i+b_i\beta'_i),$$
where $\alpha_i',\beta_i'$ are parallel copies of $\alpha_i,
\beta_i$.

Let $\phi$ be a topological Whitney disk connecting $\mathbf x$ to
$\mathbf y$. Since $\alpha_i'$ is disjoint from all
$\alpha$--curves, we have $\alpha_i'\cdot\partial_{\alpha}\phi=0$.
Similarly,
$$\beta_i'\cdot\partial_{\alpha}\phi=-\beta_i'\cdot\partial_{\beta}\phi=0.$$
We have \begin{equation}\label{eq:Diffn} n_{\mathbf
x}(B')-n_{\mathbf y}(B')=-\partial
B'\cdot\partial_{\alpha}\phi=m(\zeta_1-\zeta_2)\cdot\partial_{\alpha}\phi\in
m\mathbb Z.
\end{equation}

Pick an intersection point $\mathbf x_0$ representing the Spin$^c$
structure $\mathfrak s$. After adding copies of $\Sigma$ to $B'$,
we can assume that $n_{\mathbf x_0}(B')$ is divisible by $m$.  Since any two intersection points representing $\mathfrak s$ are
connected by a topological Whitney disk, (\ref{eq:Diffn}) implies
that $n_{\mathbf x}(B')$ is divisible by $m$ for any $\mathbf x$
representing $\mathfrak s$.

Now we define a map $H\co CF^{\infty}(Y,\mathfrak s)\to
CF^{\infty}(Y,\mathfrak s)$ by letting
$$H(\mathbf x)=\frac{n_{\mathbf x}(B')}{m}\cm\mathbf x$$
on generators, and extending linearly.  It follows from (\ref{eq:Diffn}) that
$$\curve^{\zeta_1}-\curve^{\zeta_2}=\partial\circ H-H\circ\partial.$$
Namely, $\curve^{\zeta_1},\curve^{\zeta_2}$ are chain homotopic.
\end{proof}

In light of the lemma, we will often denote the induced map  on homology by 
$\Curve^{[\zeta]}$, where $[\zeta]\in H_1(Y;\mathbb Z)/\mathrm{Tors}$.   Lemma 4.17 of \cite{OSzAnn1} shows that $\Curve^{[\zeta]}$ satisfies  $\Curve^{[\zeta]}\circ \Curve^{[\zeta]}=0$, hence varying the class within $H_1(Y;\Z)/\mathrm{Tor}$ endows the Floer homology groups with the structure of a
$\Lambda^*(H_1(Y;\mathbb Z)/\mathrm{Tors})$ module.   From its definition, the action satisfies $\Curve^{2[\zeta]}=\Curve^{[\zeta]}+\Curve^{[\zeta]}$.  In light of the fact that we  work with $\F=\Z/2\Z$ coefficients, it will thus make more sense to regard the action as by  $\Lambda^*(H_1(Y;\F))$, where it is understood that classes in $H_1(Y;\F)\cong H_1(Y;\Z)\otimes \F$  arising from even torsion in $H_1(Y;\Z)$ act as zero.

An important example is $Y=\#^n(S^1\times S^2)$.  The module structure of $\widehat{HF}(Y)$ has been computed by \ons\ \cite{OSzAnn2,OSzKnot}:
As a module over $\Lambda^*(H_1(Y;\F))$, $$\widehat{HF}(Y)\cong \Lambda^*(H^1(Y;\F)),$$
where the action of $[\zeta]$ is given by the contraction operator $\iota_{[\zeta]}$ defined using the  natural hom pairing.
We remark that the ring $\Lambda^*(H_1(Y;\F))$ is isomorphic to $\F[X_0,\dots,X_{n-1}]/(X_0^2,\dots,X_{n-1}^2)$, and the module $\Lambda^*(H^1(Y;\F))$ is isomorphic to the free module $\F[X_0,\dots,X_{n-1}]/(X_0^2,\dots,X_{n-1}^2)$.

As with the module structure on Khovanov homology, we can consider the homology of the Heegaard Floer homology with respect to $\Curve^{[\zeta]}$:
$$H_*( \HFc(Y,\spinc), \Curve^{[\zeta]}).$$ 
For $Y=\#^n(S^1\times S^2)$, we have $H_*(\HFa(Y), X_i)=0$ for all $i$.

We conclude this subsection by analyzing  the
$H_1(Y;\F)$ action in the presence of essential spheres.  We begin with the case of separating spheres.  In this case the action splits according to a K{\"u}nneth principle.  More precisely, let $$(\Sigma_i,\mbox{\boldmath${\alpha}$}_i,
\mbox{\boldmath$\beta$}_i,z_i)$$ be a Heegaard diagram for $Y_i$,
$i=1,2$. Let $\Sigma$ be the connected sum of $\Sigma_1$ and
$\Sigma_2$, with the connected sum performed at $z_1$ and $z_2$.
Let $\zeta_i\subset\Sigma_i\backslash\{z_i\}$ be a closed curve.
Suppose $\mathfrak s_1\in\mathrm{Spin}^c(Y_1), \mathfrak
s_2\in\mathrm{Spin}^c(Y_2)$. Now
$$(\Sigma,\mbox{\boldmath${\alpha}$}_1\cup\mbox{\boldmath${\alpha}$}_2,
\mbox{\boldmath$\beta$}_1\cup\mbox{\boldmath${\beta}$}_2,z_1=z_2)$$
is a Heegaard diagram for $Y_1\#Y_2$. Using the proof of the
K\"unneth formula for $\widehat{HF}$ of connected sums
\cite[Proposition~6.1]{OSzAnn2}, one sees that the action of
${\zeta_1\cup\zeta_2}$ on
$$\widehat{CF}(Y_1\#Y_2,\mathfrak s_1\#\mathfrak s_2)\cong
\widehat{CF}(Y_1,\mathfrak s_1)\otimes\widehat{CF}(Y_2,\mathfrak
s_2)$$ is given by
\begin{equation}\label{eq:ConnSum}
\curve^{\zeta_1\cup\zeta_2}=\curve^{\zeta_1}\otimes
\mathrm{id}+\mathrm{id}\otimes \curve^{\zeta_2}.
\end{equation}

Next, we turn to the case of non-separating spheres.  Here, we have a vanishing theorem for the homology of the Floer homology with respect to the action.  First we state a version of the K{\"u}nneth theorem in this context.

\begin{prop}\label{prop:OneHandle}
Let $\mathfrak s$ be a Spin$^c$ structure on a closed oriented
$3$--manifold $Y$. Let $\mathfrak s_0$ be the Spin$^c$ structure
on $S^1\times S^2$ with $c_1(\mathfrak s_0)=0$. Suppose $\zeta_1$
is a closed curve in $Y$, and $\zeta_2$ is a closed curve in
$S^1\times S^2$. Then Equation (\ref{eq:ConnSum}) holds on the chain complex
$$CF^{\circ}(Y\#(S^1\times S^2),\mathfrak s\#\mathfrak s_0)\cong
CF^{\circ}(Y,\mathfrak s)\otimes \widehat{CF}(S^1\times
S^2,\mathfrak s_0).$$ 
\end{prop}
\begin{proof}
This follows from the proof of \cite[Proposition~6.4]{OSzAnn2}.
\end{proof}

We now have the promised vanishing theorem.

\begin{prop}\label{prop:NSSphere}
Let $\sigma\co \mathbb Z\to\mathbb F$ be the natural quotient map which sends $1$ to
$\mathbf 1$. Let $[\zeta]\in H_1(Y;\mathbb Z)/\mathrm{Tors}$ be a class for which
there exists a two-sphere $S\subset Y$ satisfying
$\sigma([\zeta]\cdot[S])\ne0\in\mathbb F$. Then
$$H_*(\widehat{HF}(Y;\mathbb F),\Curve^{[\zeta]})=0\ ,\ H_*({HF}^+(Y;\mathbb F),\Curve^{[\zeta]})=0.$$
\end{prop}
\begin{proof}
Since the sphere $S$ is homologically nontrivial, it is
non-separating. Thus $Y\cong Y_1\#(S^1\times S^2)$, with $S$ isotopic to  $*\times S^2$ in $S^1\times S^2$.  We can express $[\zeta]$ as
$$[\zeta]=[\zeta_1]\oplus[\zeta_2]\in H_1(Y_1)\oplus H_1(S^1\times S^2)\cong H_1(Y),$$ where 
$\sigma([\zeta_2]\cdot[S])\ne0$. An explicit calculation with a genus one Heegaard diagram shows that $\widehat{HF}(S^1\times S^2)$ has
two generators $x,y$ for which 
$$\Curve^{[\zeta_2]}(x)=([\zeta_2]\cdot[S])y\ ,\  \Curve^{[\zeta_2]}(y)=0.$$ Hence $H_*(\widehat{HF}(S^1\times
S^2;\mathbb F),\Curve^{[\zeta_2]})=0$. Our conclusion now follows
from the previous proposition.
\end{proof}

\subsection{Holomorphic polygons and the $\Lambda^*(H_1(Y;\F))$ action }\label{subsec:polygons} 

In this subsection we consider operators on Heegaard Floer homology induced by counting pseudo-holomorphic Whitney polygons, and the interaction of these operators with the $H_1(Y;\Z)/\mathrm{Tors}$ action.  Our main result here is a compatibility relation, Theorem \ref{thm:AinfinityH1}.    This relation will be useful in two ways.  First, it will allow us to understand  the $H_1(Y;\Z)/\mathrm{Tors}$ action in the context of the cobordism maps  constituting the $(3+1)$--dimensional TQFT in \os \ theory.   Second, it will be the key ingredient for showing that the spectral sequence from Khovanov homology to the Heegaard Floer homology of the branched double cover respects the module structure on both theories.

Let $(\Sigma, \etab^0,...,\etab^n,z)$ be a pointed Heegaard $(n+1)$--tuple diagram;  that is, a surface of genus $g$, together with $(n+1)$ distinct $g$--tuples $\etab^i=\{\eta_1^i,...,\eta_g^i\}$ of homologically linearly independent attaching curves as in  \cite[Section 4.2]{OSzBrCov}.   A Heegaard $(n+1)$--tuple diagram naturally gives rise to an oriented $4$--manifold $W_{0,...,n}$ by the pants construction  \cite[Subsection 8.1]{OSzAnn1}.  The boundary of this manifold is given by:
$$\partial W_{0,...,n}= -Y_{0,1}\sqcup -Y_{1,2} \sqcup ... \sqcup -Y_{n-1,n}\sqcup Y_{0,n}, $$ where $Y_{i,j}$ is the $3$--manifold specified by the attaching curves $\etab^i,\etab^j$.   Each $\etab^i$ gives rise to a Lagrangian torus in Sym$^g(\Sigma)$ which we denote $\T_i$. We also require an admissibility hypotheses on two-chains with boundary that realize  homological relations between curves in the different $g$--tuples $\etab^i$, the so-called {\em multi-periodic domains}.   For this section it suffices to assume that all  multi-periodic domains have positive and negative coefficients, a condition which can be achieved by winding their boundary on the Heegaard diagram.  

To a Heegaard $(n+1)$--tuple diagram one can associate a chain map:
$$ f_{0,...,n}: \bigotimes_{i=1}^{n} \CFa(Y_{i-1,i}) \longrightarrow  \CFa(Y_{0,n}),$$
defined by counting pseudo-holomorphic maps of Whitney $(n+1)$--gons into Sym$^g$ with boundary conditions in the Lagrangian tori. On generators,  $f_{0,...,n}$ is given by
$$ f_{0,...,n}(\x_1\otimes...\otimes \x_{n})=\sum_{\y}
\sum_{\underset{\mu(\phi)=2-n, \ n_z(\phi)=0}{\phi\in\pi_2(\x_1,,...,\x_{n},\y)}}{  \#{ \mathcal M}}(\phi)\cm \y,$$
and we extend linearly to the full complex.  

Let us unpack the notation a bit.  The first summation is over all $\y\in\mathbb T_{0}\cap\mathbb T_{n}$.  In the second, we sum over all homotopy classes of Whitney $(n+1)$--gons.  We  restrict attention to only those homotopy classes with Maslov index $\mu(\phi)=2-n$.   This condition ensures that as we vary within the $(n-2)$--dimensional universal family of conformal $(n+1)$--gons, we obtain a finite number of pseudo-holomorphic maps from these  conformal polygons to Sym$^g(\Sigma)$.  We denote this number, reduced modulo $2$, by $\#\mathcal{M}(\phi)$.    When $n=1$ the domain of our pseudo-holomorphic map has a $1$--dimensional automorphism group isomorphic to $\R$, and we consider instead the unparametrized moduli space $\widehat{\mathcal M}={\mathcal M}/\R$.   See \cite[Section 8]{OSzAnn1} and \cite[Section 4]{OSzBrCov} for more details.

 The polygon operators satisfy an $A_\infty$-associativity relation (see \cite[Equation (8)]{OSzBrCov}):
$$ \sum_{0\le i<j\le n} f_{0,...i,j,...n} \ (\theta_1\otimes ...\otimes f_{i,...,j}(\theta_{i+1}\otimes ... \otimes \theta_{j})\otimes...\otimes \theta_{n})=0,$$
where $\theta_i\in \CFa(Y_{i-1,i})$ are chains in the Floer complexes associated to the pairs of Lagrangians assigned to the vertices of the $(n+1)$--gon.   This relation breaks up into a collection of relations, one for each $n>0$.  For $n=1$, the relation simply states that $f_{0,1}: \CFa(Y_{0,1})\rightarrow \CFa(Y_{0,1})$ is a differential. 

Examining the definition of the $H_1(Y;\Z)/\mathrm{Tor}$ action from the previous section, we see that it closely resembles the Floer differential.    Indeed, the action of $\zeta$ is simply the Floer differential weighted by   $\zeta\cm \partial_0 \phi$, the intersection of a curve representing a class in  $H_1(Y;\Z)/\mathrm{Tor}$ with the $\etab^0$--component of the boundary of the image of a pseudo-holomorphic Whitney disk.  Motivated by this similarity we define operators for any closed curve $\zeta\in \Sigma$
$$ 
\curve^\zeta_{0,...,n}\co \bigotimes_{i=1}^{n} \CFa(Y_{{i-1},{i}}) \longrightarrow  \CFa(Y_{{0},{n}}).$$
On generators,  $\curve^\zeta_{0,...,n}$ is given by
$$ \curve^\zeta_{0,...,n}(\x_1\otimes...\otimes \x_{n})=\sum_{\y}
\sum_{\underset{\mu(\phi)=2-n, n_z(\phi)=0}{\phi\in\pi_2(\x_1,...,\x_{n},\y)}}{ (\zeta\cm(\partial_0\phi)) \#{ \mathcal M}}(\phi)\cm \y,$$
and we extend linearly to the full complex.    We have the following $A_\infty$-compatibility relation between the $\curve^\zeta$ and $f$ maps.  

\begin{thm}\label{thm:AinfinityH1} Let $\zeta\in \Sigma$ be a curve on the Heegaard surface of a given $(n+1)$--tuple diagram which is in general position with respect to all the attaching curves.  Let  $f_{0,...n}$ and $\curve^\zeta_{0,...n}$ be the polygon maps defined above.  Then we have the following relation:
$$ \sum_{0\le i<j\le n} \curve^\zeta_{0,...i,j,...n} \ (\theta_1\otimes ...\otimes f_{i,...,j}(\theta_{i+1}\otimes ... \otimes \theta_{j})\otimes...\otimes \theta_{n}) \ \ + $$
$$  \sum_{0< j\le n} f_{0,j,...n} \ (\curve^\zeta_{0,...j}(\theta_1\otimes ...\otimes \theta_{j})\otimes \theta_{j+1}\otimes ... \otimes \theta_{n})=0,$$
Where $\theta_i\in \CFa(Y_{i-1,i})$ are any Heegaard Floer chains.
\end{thm}

\begin{proof}  We consider the $n=3$ version of the relation. A guiding schematic is shown in Figure \ref{fig:endmoduli}. First we establish some  notation. Suppose   $\theta=\sum_{l=1}^{N} \x_l$ is a chain (recall we work mod 2, so that no coefficient is needed in front of $\x_i$).    Let $$\pi_2(\theta,\p,\q,\s):=\coprod_{l=1,...,N} \{\phi\in \pi_2(\x_l,\p,\q,\s)  \},$$ with similar notation to handle $\pi_2(\p,\theta,\q,\s)$, $\pi_2(\s,\p,\theta,\q)$, and $\pi_2(\p,\q,\s,\theta)$

 Consider the moduli space
$$\mathcal M(\phi\in \pi_2(\theta_1,\theta_2,\theta_3,\theta_4), \mu(\phi)=0):=\coprod_{\underset{\mu(\phi)=0, n_z(\phi)=0}{ \phi\in\pi_2(\theta_1,\theta_2,\theta_3,\theta_4) } } \mathcal M\big(\phi \big).$$ For an appropriate family of almost complex structures on Sym$^g(\Sigma)$, this moduli space is a smooth $1$--dimensional manifold with a Gromov compactification.  
The boundary points of this
compactification are in bijection with points in the following products of zero dimensional moduli spaces:
$$\widehat{\mathcal M}\big(\phi_1\in\pi_2(\theta_1,\rho),\mu(\phi_1)=1\big)\times
\mathcal M\big(\psi_1\in\pi_2(\rho,\theta_2,\theta_3,\theta_4),\mu(\psi_1)=-1\big),$$
$$\widehat{\mathcal M}\big(\phi_2\in\pi_2(\theta_2,\rho),\mu(\phi_2)=1\big)\times
\mathcal M\big(\psi_2\in\pi_2(\theta_1,\rho,\theta_3,\theta_4),\mu(\psi_2)=-1\big),$$
$$\widehat{\mathcal M}\big(\phi_3\in\pi_2(\theta_3,\rho),\mu(\phi_3)=1\big)\times
\mathcal M\big(\psi_3\in\pi_2(\theta_1,\theta_2,\rho,\theta_4),\mu(\psi_3)=-1\big),$$
$$\mathcal M\big(\phi_4\in\pi_2(\theta_1,\theta_2,\theta_3,\rho),\mu(\phi_4)=-1\big) \times \widehat{\mathcal M}\big(\psi_4\in\pi_2(\rho,\theta_4),\mu(\psi_4)=1\big)
,$$
and 
 $${\mathcal M}\big(\phi_5\in\pi_2(\theta_2,\theta_3,\rho),\mu(\phi_5)=0\big)\times
\mathcal M\big(\psi_5\in\pi_2(\theta_1,\rho,\theta_4),\mu(\psi_5)=0\big),$$
$${\mathcal M}\big(\phi_6\in\pi_2(\theta_1,\theta_2,\rho),\mu(\phi_6)=0\big)\times
\mathcal M\big(\psi_6\in\pi_2(\rho,\theta_3,\theta_4),\mu(\psi_6)=0\big).$$  In each case $\rho$ is dummy variable ranging over all chains in the appropriate Floer complex.  Thus we are considering moduli spaces associated to all possible decompositions of the homotopy classes in $\pi_2(\theta_1,\theta_2,\theta_3,\theta_4)$ with $\mu(\phi)=0$ and $n_z(\phi)=0$ into homotopy classes of bigons concatenated with rectangles and triangles concatenated with triangles. For each $\phi\in \pi_2(\theta_1,\theta_2,\theta_3,\theta_4)$ with $\mu(\phi)=0$ we expand the equality
$$(\zeta\cdot\partial_{0}\phi)\#(\partial\mathcal M(\phi))=0,$$ 
and observe that since $\phi_i*\psi_i=\phi$ for each $i=1,...,6$, we have
$$\zeta\cdot\partial_{0}\phi=\zeta\cdot\partial_{0}\phi_i+\zeta\cdot\partial_{0}\psi_i.$$

\begin{figure}
\centering
\def\svgwidth{4.5in}
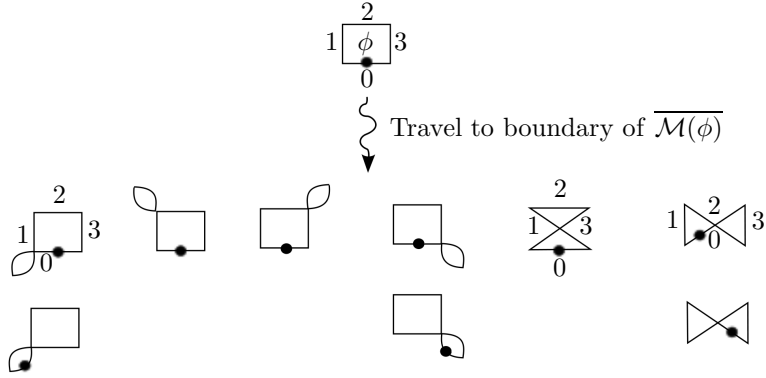
\caption{\label{fig:endmoduli}Schematic for $n=3$  case of Theorem \ref{thm:AinfinityH1}. One dimensional moduli spaces of rectangles are considered, with boundary conditions in the numbered Lagrangians. The compactified  moduli spaces have six types of boundary points.  A dark circle on a boundary arc labeled ``$0$" indicates that we weight the count of points in a homotopy class by  the intersection of the $\etab^0\subset \Sigma$ portion of the boundary  with the curve $\zeta\subset \Sigma$.   Conservation of intersection number ensures that the sum of the weights in any vertical column equals $\zeta\cm\partial_0\phi$.  The  $9$ types of weighted boundary point counts are in bijection with the terms in the $n=3$ $A_\infty$-relation for the $\zeta$-action.  }

\end{figure}

\noindent Summing over all $\phi\in \pi_2(\theta_1,\theta_2,\theta_3,\theta_4)$, the terms in the expansion   correspond to the coefficient of $\theta_4$ in 
$$ \curve^\zeta_{0,1,2,3}(f_{0,1}(\theta_1)\otimes\theta_2\otimes\theta_3)+f_{0,1,2,3}(\curve^\zeta_{0,1}(\theta_1)\otimes \theta_2\otimes \theta_3) 
$$ $$+ \curve^\zeta_{0,1,2,3}(\theta_1\otimes f_{1,2}(\theta_2)\otimes\theta_3) $$ $$+
\curve^\zeta_{0,1,2,3}(\theta_1\otimes\theta_2\otimes f_{2,3}(\theta_3))$$ $$+ \curve^\zeta_{0,3}(f_{0,1,2,3}(\theta_1\otimes \theta_2\otimes \theta_3)) + f_{0,3}(\curve^\zeta_{0,1,2,3}(\theta_1\otimes \theta_2\otimes \theta_3)) $$ $$ +
\curve^\zeta_{0,1,3}(\theta_1\otimes f_{1,2,3}(\theta_2\otimes\theta_3)) $$  $$ +
\curve^\zeta_{0,2,3}(f_{0,1,2}(\theta_1\otimes\theta_2)\otimes\theta_3) 
 + f_{0,2,3}(\curve^\zeta_{0,1,2}(\theta_1\otimes \theta_2)\otimes \theta_3).$$ The terms in the six lines correspond to the contribution from the six types of boundary points for the $1$--dimensional moduli space, listed above.   Note that the asymmetry between the $a\circ f$ and $f\circ a$ terms is a result of the fact that      $$\zeta\cm \partial_0 \phi_2=\zeta\cm \partial_0\phi_3=\zeta\cm \partial_0\phi_5=0,  $$ which, in turn, is due to the fact that these homotopy classes have boundary conditions outside of $\T_0$.
This proves the proposition for $n=3$.  The general case is a straightforward yet notationally cumbersome extension.
\end{proof}

\subsection{Cobordisms} \label{sec:Cob}
In this section we use the compatibility relation between the polygon operators $\curve^{\zeta}$ and $f$ given by Theorem \ref{thm:AinfinityH1} to understand the behavior of the \os \ cobordism maps with respect to the $H_1(Y;\Z)/\mathrm{Tors}$ action.  Our main goal is Theorem \ref{thm:Cob}, which shows that the maps on Floer homology induced by a cobordism $W$ from $Y_1$ to $Y_2$ commute with the action, provided that the curves in $Y_1$ and $Y_2$ inducing the action are homologous in $W$.

The key is the $n=2$  compatibility relation.  Let $(\Sigma,\etab^0,\etab^1,\etab^2,z)$
be a pointed Heegaard triple-diagram,  and let 
$$f_{0,1,2}\co
{\CFa}(Y_{0,1}) \otimes
{\CFa}(Y_{1,2})\to
{\CFa}(Y_{0,2})$$ denote
the polygon map from the previous section, which counts pseudo-holomorphic Whitney triangles.  Let $\Theta \in {\CFa}(Y_{1,2})$ be a cycle. Then we define a map
 $f_{\Theta}: \CFa(Y_{0,1})\longrightarrow \CFa(Y_{0,2})$ by $$ f_{\Theta}(-):=f_{0,1,2}(-\otimes\Theta).$$  The $A_\infty$ relation for $n=2$ states that
\begin{eqnarray*}
 &&f_{0,1,2}(f_{0,1}(-)\otimes \Theta)+ f_{0,1,2}(-\otimes f_{1,2}(\Theta))+ f_{0,2}(f_{0,1,2}(-\otimes\Theta) \\
&=& f_{\Theta}\circ \partial_{0,1}+ 0 + \partial_{0,2}\circ f_{\Theta}=0,
\end{eqnarray*}
 where we have used the fact that $\Theta$ is a cycle for the $f_{1,2}$  differential.  This shows that $f_\Theta$ is a chain map. 
Note that $Y_{1,2}$ could be any $3$--manifold for the moment, and $\Theta$ an arbitrary cycle.   In many concrete situations $Y_{1,2}\cong \#^n S^1\times S^2$ and $\Theta$ can be assumed to be a generator, $\Theta={\mathbf y}\in\mathbb T_{1}\cap \mathbb T_{2}$.

Given a closed curve $\zeta\subset \Sigma$ that is in general position with respect to the attaching curves, we can regard $\zeta$ as a curve in $Y_{0,1}$ and $Y_{0,2}$.   As such, we have an action by $\zeta$ on $\CFa(Y_{0,1})$ and $\CFa(Y_{0,2})$.     The following lemma shows that the induced maps on homology commute with that of $f_{\Theta}$.

\begin{lem}\label{lem:Af=fA}  Given a pointed Heegaard triple diagram, $(\Sigma,\etab^0,\etab^1,\etab^2,z)$, let $f_{\Theta}=f_{0,1,2}(-\otimes \Theta)$ be the chain map from $\CFa(Y_{0,1})$ to $\CFa(Y_{0,2})$ associated to a cycle $\Theta\in \CFa(Y_{1,2})$.  Let $\zeta\subset \Sigma$ be a curve and $\curve^{\zeta_0},\curve^{\zeta_1}$ denote the actions of $\zeta$ on
$\CFa(Y_{0,1})$ and $\CFa(Y_{0,2})$, respectively.
Then $f_{\Theta}\circ \curve^{\zeta_0}$ and
$\curve^{\zeta_1}\circ f_{\Theta}$ are chain homotopic.
\end{lem}

\begin{proof}  The $n=2$ version of Theorem \ref{thm:AinfinityH1} states that $$ \curve_{0,1,2}^\zeta(f_{0,1}(-)\otimes \Theta)+  \curve_{0,1,2}^\zeta(-\otimes f_{1,2}(\Theta))+\curve_{0,2}^\zeta(f_{0,1,2}(-\otimes \Theta)) +$$ $$+ f_{0,1,2}(\curve_{0,1}^\zeta(-)\otimes\Theta)+ f_{0,2}(\curve_{0,1,2}^\zeta(-)\otimes \Theta)=0.$$ Let $H(-):=\curve_{0,1,2}^\zeta(-\otimes\Theta)$.  Using the fact that $f_{1,2}(\Theta)=0$, we get 
$$H\circ\partial_{0,1}  +\curve^{\zeta_1}\circ f_\Theta+ f_\Theta\circ \curve^{\zeta_0}+ \partial_{0,2}\circ H=0.$$ In other words, $\curve_{0,1,2}^\zeta$ provides the requisite   chain homotopy. 
\end{proof}

The following theorem concerns the naturality of the homological
action under the homomorphisms induced by cobordisms. Similar
results can be found in \cite{OSzFour}.

\begin{thm}\label{thm:Cob}
Suppose $Y_1,Y_2$ are two closed, oriented, connected
$3$--manifolds, and $W$ is a cobordism from $Y_1$ to $Y_2$. Let
$$\widehat{F}_{W}\co \HFa(Y_1)\longrightarrow
\HFa(Y_2)$$ be the homomorphism induced by
$W$. Suppose $\zeta_1\subset Y_1$, $\zeta_2\subset Y_2$ are two
closed curves which are homologous in $W$. Then 
$$\widehat{F}_{W}\circ\Curve^{[\zeta_1]}= \Curve^{[\zeta_2]}\circ\widehat{F}_{W}.$$
\end{thm}
\begin{proof}
Since $\zeta_1$ and $\zeta_2$ are homologous in $W$, there exists
an oriented proper surface $S\subset W$ connecting $\zeta_1$ to
$\zeta_2$. By adding tubes between the components of $S$, we can assume
that it is connected.  Let $S'$ denote the surface obtained by removing a collar neighborhood of the components of the boundary of $S$ lying in $Y_2$. 

Now let $W_1'$ be a neighborhood of $Y_1\cup S'$ in $W$. Then
$W_1'$ can be obtained from $Y_1\times I$ by adding $1$--handles.  
Thus $\partial W_1'=-Y_1\sqcup Y_1'$, where $Y_1'\cong Y_1\#^kS^1\times S^2$.   The boundary of $S'$ induces a curve $\zeta_1'\subset Y_1'$.
Using Lemma~\ref{homologyaction} and Proposition~\ref{prop:OneHandle}, we see that the conclusion of the theorem 
holds for the cobordism $W_1'$ and the curves $\zeta_1,\zeta_1'$.
                                                      
The cobordism $W$ can thus be decomposed as $W_1'\cup_{Y_1'}W_2$, where
$W_2$ consists of $1$--, $2$--, and $3$--handles which
are disjoint from $S$; indeed, the portion of $S$ in $W_2$ is simply the trace of $\zeta_1'$  in this cobordism. The map on Floer homology induced by a cobordism is defined by associating chain maps to handle attachments in a handle decomposition \cite{OSzFour}.  As above, Lemma \ref{homologyaction}  and Proposition~\ref{prop:OneHandle} show that the maps associated to the $1$-- and $3$--handles in $W_2$ commute with the action.   Commutativity of the action with the maps associated to the $2$--handles is ensured by Lemma~\ref{lem:Af=fA} (the $2$--handle maps are defined as $f_\Theta$ for an appropriate Heegaard triple diagram and choice of $\Theta$). Hence our
conclusion for $W$ holds by the composition law for cobordism maps, \cite[Theorem 3.4]{OSzFour}.
\end{proof} 

\begin{rem} The above theorem has obvious generalizations in two directions.  First, one could refine the theorem to account for $\SpinC$ structures: to a cobordism $W$ equipped with a $\SpinC$ structure $\spinct$ we have maps between the Floer homology groups $\HFa(Y_1,\spinct|_{Y_1})$ and $\HFa(Y_2,\spinct|_{Y_2})$, and one can show that the commutativity theorem respects this structure.  Second, one can prove the same theorem for the other versions of Floer homology.  In this situation, however, one must take care.  We have been considering the sum of cobordism maps associated to all $\SpinC$ structures on $W$.  This cannot, in general, be done with the minus and  infinity versions of Floer homology, since these versions require admissibility hypotheses that cannot be simultaneously achieved for all $\SpinC$ structures with a single Heegaard diagram.  For that reason, the commutativity theorem in these versions must incorporate  $\SpinC$ structures.  Alternatively, one can continue to sum over all $\SpinC$ structures if we consider the ``completed" versions of $\CFm$ and $\CFinf$ with base rings $\F[[U]]$ and $\F[[U,U^{-1}]$, respectively.  Similar remarks hold for the $A_\infty$ compatibility relation, Theorem \ref{thm:AinfinityH1}.
\end{rem}

\subsection{Heegaard Floer homology with twisted
coefficients}\label{sec:Twist}

Suppose $Y$ is a closed oriented $3$--manifold, $\mathfrak
s\in\mathrm{Spin}^c(Y)$. Let
$$(\Sigma,\mbox{\boldmath${\alpha}$},
\mbox{\boldmath$\beta$},z)$$ be an admissible Heegaard diagram for
$(Y,\mathfrak s)$. Let $\mathcal
R=\mathbb F[T,T^{-1}]$. 

Given $[\omega]\in H^2(Y;\mathbb Z)$, let $\eta\subset\Sigma$ be a
closed curve which represents the Poincar\'e dual of $\omega$ in
$H_1(Y;\mathbb Z)$. Let $\underline{CF}^{\infty}(Y,\mathfrak
s;\mathcal R_{\eta})$ be the $\mathcal R[U,U^{-1}]$--module freely generated
by $\mathbf x\in\mathbb
T_{\alpha}\cap\mathbb T_{\beta}$ which represent $\mathfrak s$. The differential $\underline{\partial}$ is
defined by
$$\underline{\partial}(\x)=\sum_{\mathbf
y}\sum_{\stackrel{\scriptstyle\phi\in\pi_2(\mathbf x,\mathbf y)}
{\mu(\phi)=1}}\#\big(\mathcal M(\phi)/\mathbb R\big)
T^{\eta\cdot\partial_{\alpha}\phi}U^{n_z(\phi)}\y.$$ The
homology of this chain complex depends on $\eta$ only through its
homology class in $H_1(Y)$. Similarly, there are chain complexes
$\underline{CF}^{\pm}(Y,\mathfrak s;\mathcal R_{\eta})$ and
$\underline{\widehat{CF}}(Y,\mathfrak s;\mathcal R_{\eta})$. Their
homologies are called the $\omega$--twisted Floer homologies,
denoted $$\underline{HF}^{\circ}(Y,\mathfrak s;\mathcal R_{\eta})
\quad\text{or}\quad \underline{HF}^{\circ}(Y,\mathfrak s;\mathcal
R_{[\omega]})$$ when there is no confusion.

The field $\mathbb F=\mathcal R/(T-1)\mathcal R$ is also an $\mathcal
R$--module. By definition
$$CF^{\circ}(Y;\mathbb F)=\underline{CF}^{\circ}(Y;\mathcal R_{[\omega]})\otimes_{\mathcal R}\mathbb F$$
is the usual untwisted Heegaard Floer chain complex over $\mathbb
F$.

There are chain maps on the $\omega$--twisted chain complex
induced by cobordisms, as in Ozsv\'ath--Szab\'o
\cite{OSzAnn1,OSzAnn2}, Jabuka--Mark \cite{JM} and Ni
\cite{NiNSSphere}. More precisely, suppose $W\co Y_1\to Y_2$ is a
cobordism, $[\Omega]\in H^2(W;\mathbb Z)$. Let
$[\omega_1],[\omega_2]$ be the restriction of $[\Omega]$ to $Y_1$
and $Y_2$, respectively. Then there is a map
$$\underline{F}^{\circ}_{W;[\Omega]}\co\underline{HF}^{\circ}(Y_1;\mathcal R_{[\omega_1]})\to
\underline{HF}^{\circ}(Y_2;\mathcal R_{[\omega_2]}).$$

We can also define the $\Lambda^*(H_1(Y)/\mathrm{Tors})$ action on
the $\omega$--twisted Floer homology by letting
$$\curve^{\zeta}(\x)=\sum_{\mathbf y\in\mathbb T_{\alpha}\cap\mathbb T_{\beta}}
\sum_{\{\phi\in\pi_2(\mathbf x,\mathbf
y)|\mu(\phi)=1\}}\big(\zeta\cdot(\partial_{\alpha}\phi)\big)\:\#\widehat{\mathcal M}(\phi)
T^{\eta\cdot\partial_{\alpha}\phi}U^{n_z(\phi)}\y.$$
As in
Subsection~\ref{sec:Cob}, there are twisted versions of
(\ref{eq:ConnSum}) and Theorem~\ref{thm:Cob}.

\section{Module structures and the link surgeries spectral sequence}

The connection between Khovanov homology and Heegaard Floer homology arises from a calculational tool called the {\em link surgeries spectral sequence}, \cite[Theorem~4.1]{OSzBrCov}. Roughly, this device takes as input a framed link  in a $3$--manifold, with output a  filtered complex.   The homology of the complex is isomorphic to the  Heegaard Floer homology of the underlying $3$--manifold, and the $E_1$ term of the associated spectral sequence splits as a direct sum of the Heegaard Floer groups of the manifolds obtained by surgery on the link with varying surgery slopes. The differentials in the spectral sequence are induced by the holomorphic polygon maps of Subsection~\ref{subsec:polygons}.   Applying this machinery to a particular surgery presentation of the $2$-fold branched cover of a link produces the spectral sequence from Khovanov homology to the Heegaard Floer homology.

Our purpose in this section is to refine the link surgeries spectral sequence to incorporate the $\HFmod$--module structure on Floer homology, and subsequently relate this structure to the module structure on Khovanov homology.   In Subsection   \ref{subsec:lsss}  we prove that a curve in the complement of a framed link induces a filtered chain map acting on the complex giving rise to the link surgeries spectral sequence, Theorem \ref{thm:ssnaturality}.  In Subsection \ref{subsec:khss} we use these filtered chain maps to endow the  spectral sequence from Khovanov homology to the Floer homology of the branched double cover with a module structure.  This  module structure allows us to prove a ``collapse" result for the spectral sequence (Proposition \ref{collapse}).  This result  states that if the Khovanov module of a link is isomorphic to that of the unlink then this module is also isomorphic to the Floer homology of the branched double cover.   Since our results involve filtered chain maps and the morphisms they induce on   spectral sequences, we begin with a digression on spectral sequences and their morphisms.  The reader familiar with these concepts may wish to skip ahead to Subsection \ref{subsec:lsss}, but  is warned that our perspective on spectral sequences, while equivalent to the standard treatment, is slightly non-standard.

\subsection{A review of spectral sequences and their morphisms}\label{subsec:ssconstruction}

Suppose that a complex  of vector spaces $(C,d)$ admits a decomposition  \begin{equation}\label{eq:decomp} C=\underset{i\ge 0}\bigoplus\ C^{(i)},
\end{equation}
which is respected by the differential, in the sense that $d=d^{(0)}+d^{(1)}+d^{(2)}+...$, where $d^{(m)}: C^{(i)}\rightarrow C^{(i+m)}$ for each $i$.  From this structure, one can construct a spectral sequence; that is, a sequence of chain complexes $\{(E_r,\delta_r)\}_{r=0}^{\infty}$ satisfying $H_*(E_r,\delta_r)\cong E_{r+1}$.   Under mild  assumptions, one has $\delta_r=0$ for all sufficiently large $r$, and the resulting limit satisfies  $E_\infty\cong H_*(C,d)$.    Typically, one constructs such a spectral sequence by noting that the  subcomplexes, $$F^p= \underset{i\ge p} \bigoplus\ C^{(i
)},$$
form a filtration $C=F^0\supset F^1\supset ...,$ and then appealing to the well-known construction of the spectral sequence associated to a filtered complex  (see, e.g. \cite{McCleary}). 

For the purpose of understanding morphisms of spectral sequence induced by filtered chain maps, we find it more transparent to construct the spectral sequence by a method called  {\em iterative cancellation} or {\em reduction}, a procedure we  briefly recall.  The method relies on the following well-known lemma:

\begin{lem}\label{lem:cancel}  Let $(C,d)$ be a chain complex of $R$-modules, freely generated by chains  $\{\bf x_i\}_{i\in I}$, and suppose that $d({\bf x_k}, {\bf x_l})=1$, where $d({\bf a},{\bf b})$ denotes the coefficient of $\bf b$ in $d(\bf a)$.    Then we can define a complex $(C',d')$,  freely generated by $\{{\bf x_i} | i\ne k,l\}$, which is chain homotopy equivalent to $(C,d)$.  
\end{lem}
\begin{proof} Let $h:C\rightarrow C$ be the module homomorphism defined by $h(\x_l)=\x_k$ and $h(\x_i)=0$ if $i\ne l$.  Then the differential on $C'$ is given by $$d'= \pi\circ(d-d h d)\circ \iota,$$ where $\pi: C\rightarrow C'$ and $\iota: C'\rightarrow C$ are the natural projection and inclusions.  
Now the chain maps  $f:C\rightarrow C'$ and $g:C'\rightarrow C$  defined by $$f=\pi\circ( \mathbb{I} - d\circ h)  \ \ \ \ \ g=(\mathbb{I} - h\circ d)\circ \iota, $$ are mutually inverse chain homotopy equivalences.  Indeed, $f\circ g = \mathbb{I}_{C'}$ and $g\circ f\sim \mathbb{I}_{C}$ via the homotopy $h$.
\end{proof}
\noindent Of course  we can employ the lemma under the weaker assumption that $d({\bf x_k}, {\bf x_l})$ is a unit in $R$, simply by rescaling the basis.  In the present situation $(C,d)$ is  a complex of vector spaces, so this applies whenever $d({\bf x_k}, {\bf x_l})\ne 0$.

We will use the lemma in a filtered sense.  To make this precise, given a  complex  with increasing filtration as above, let $\Filt(a)\in \Z^{\ge0}$ denote the filtration level of a chain, i.e. 
$$\Filt(a)= \text{max}\{ i\in \Z^{\ge 0} \ | a\in F^i\},$$
and define $\Filt(a,b)=\Filt(a)-\Filt(b)$.  Note that  the filtration of a linear combination of chains satisfies $$\Filt(a+b)\ge \text{min}\{ \Filt(a),\Filt(b)\}.$$

\begin{lem}\label{lem:filteredcancel} With the notation from Lemma \ref{lem:cancel}, suppose that $(C,d)$ is a filtered complex,  that $d(\x_k,\x_l)=1$ and that
$$ \Filt(d(\x_k), \x_l)\ge 0.$$
Then the reduced complex $(C',d')$ inherits a filtration $\Filt'$ from $(C,d)$ by the formula $\Filt'(a):=\Filt(\iota (a))$ and the filtration degree of $d'$ is no less than that of $d$ in the sense that
$$ \Filt'(d'(a))\ge \Filt(d\circ\iota(a)),  \quad \text{for all\ } a\in C'.$$
Moreover, if $\Filt(\x_k,\x_l)=0$, then $(C,d)$ and $(C',d')$ are filtered chain homotopy equivalent.
\end{lem}
\begin{proof}
To prove that $\Filt'$ defines a filtration on the reduced complex we must show
$$ \Filt'(a,d'(a))\le 0 \quad \text{for all\ } a \in C',$$
which, by  definition, is the same as showing
\begin{equation}\label{eq:filteredreduction1} \Filt(\iota(a))\le \Filt(\iota(d'(a))=\Filt(\iota\circ\pi\circ(d-dhd)\circ\iota(a)).\end{equation}
To begin, observe that for any $x\in C$ we have $\Filt(\iota\circ\pi(x))\ge \Filt(x)$; that is, dropping $\x_k$ and $\x_l$ from a chain can only increase the filtration level.  Thus the right-hand side of \eqref{eq:filteredreduction1} satisfies
\begin{equation}\label{eq:filteredreduction2}\Filt(\iota\circ\pi\circ(d-dhd)\circ\iota(a))\ge \Filt((d-dhd)\circ\iota(a)).\end{equation}
Let $r\in R$ denote $d(\iota(a),\x_l)$, the coefficient of $\x_l$ in $d(\iota(a))$.  By the definition of $h$, the right hand of \eqref{eq:filteredreduction2} is equal to $\Filt(d(\iota(a))-r\cm d(\x_k))$, which satisfies the inequality:
\begin{equation}\label{eq:filteredreduction3}
\Filt(d(\iota(a))-r\cm d(\x_k))\ge \text{min}\{\Filt(d(\iota(a)),\Filt(rd(\x_k))\}
\end{equation}
Now if  $\Filt(rd(\x_k))\ge \Filt(d(\iota(a))$ then the right hand side of \eqref{eq:filteredreduction3} equals $\Filt(d(\iota(a))$, which satisfies $$\Filt(d(\iota(a))\ge \Filt(\iota(a))$$ because $(C,d)$ is filtered.  If  $\Filt(rd(\x_k))<\Filt(d(\iota(a))$, then $r\ne 0$ and the right side of \eqref{eq:filteredreduction3} equals $\Filt(rd(\x_k))$ which satisfies $\Filt(rd(\x_k))\ge \Filt(\x_l)$ by assumption.  In this case, though, we have $$\Filt(\x_l)\ge \Filt(\iota(a))$$ since $d(\iota(a),\x_l)\ne 0$ and $(C,d)$ is filtered.  Thus in both cases, we have verified \eqref{eq:filteredreduction1}. The statement about filtration degree is straightforward.

To see that $C$ and $C'$ are filtered homotopy equivalent when the filtration levels of the cancelled chains agree, it suffices to show that the chain maps $f$ and $g$ and chain homotopy $h$ from the previous lemma are filtered, in the sense that $$ \Filt' \circ f\ge  \Filt,  \quad  \Filt \circ g \ge \Filt', \quad \text{and} \quad \Filt \circ h\ge \Filt.$$ These are immediate from the definition of $\Filt'$ and the fact that $h$ preserves $\Filt$ (if the value of $h$ is nonzero) when $\Filt(\x_k,\x_l)=0$.
\end{proof}

We can use the lemmas to easily produce a spectral sequence.  To begin, define $(E_0, d_0)= (C,d)$.  Now use Lemma   \ref{lem:cancel} to cancel all of the non-zero terms in the differential of order zero i.e. the $d^{(0)}$ terms.  Lemma \ref{lem:filteredcancel} implies that the result is a {\em filtered} chain homotopy equivalent complex, $(E_1,d_1)$ for which the lowest order terms in the differential are of order one; that is, the differential can be written as $d_1=d_1^{(1)}+ d_1^{(2)} +\cdots$ with respect to the natural direct sum decomposition of $E_1$ induced by \eqref{eq:decomp}.    Now cancel  the $d_1^{(1)}$ components of $d_1$.   The result is a chain homotopy equivalent complex $(E_2,d_2)$ satisfying $d_2=d_2^{(2)}+d_2^{(3)}+\cdots$.    Assuming the complex is finitely generated in each homological degree, we can iterate this process until all differentials of all orders have been cancelled.   Denote  the lowest order term of $d_r$ by $$\delta_r:= d_r^{(r)}.$$ The fact that   $d_r\circ d_r=0$ implies  $\delta_r\circ \delta_r=0$, and  canceling $d_r^{(r)}$ is  equivalent to taking homology with respect to $\delta_r$.  We represent the process schematically:
$$
\begin{array}{cc} 
 E_0, d_0=d^{(0)}+d^{(1)}+d^{(2)}+d^{(3)}+\cdots  & \\
 \downarrow &  \text{Homology with respect to }\delta_0=d^{(0)}   \\
E_1, d_1= \ \ 0  \ + d_1^{(1)} +d_1^{(2)}+d_1^{(3)}+\cdots  & \\
\downarrow & \text{Homology with respect to }\delta_1=d_1^{(1)}\\
E_2, d_2= \ \ 0  \ +  \ \ 0  \  +d_2^{(2)}+d_2^{(3)}+\cdots  & \\
\downarrow & \text{Homology with respect to }\delta_2=d_2^{(2)}\\
\vdots & \\
E_\infty, d_\infty\equiv \ \ 0 \ \  \hskip1.3in & \\
\end{array}
$$
The resulting structure is a spectral sequence with $r$-th page given by $(E_r,\delta_r)$.  By construction, $E_\infty\cong H_*(C,d)$, since $(E_\infty,d_\infty)$ is chain homotopy equivalent to $(C,d)$.  The concerned reader can take comfort in the knowledge that the spectral sequence described here is isomorphic to that produced by the standard construction, \cite[Theorem 2.6]{McCleary}.  The proof  of equivalence is straightforward but rather notationally cumbersome, and since  our results make no use of it we leave it for the interested reader.

Let $C$ and $\overline{C}$ be complexes with filtrations induced by decompositions: $$ C=\underset{i\ge 0}\bigoplus\ C^{(i)}   \ \ \ \ \ \overline{C}= \underset{i\ge 0}\bigoplus\ {\overline{C}}^{(i)},$$ and let $a:C\rightarrow \overline{C}$ be a filtered chain map, i.e. $a(F^i)\subset \overline{F}^i$.   Such a map is well-known to induce a morphism between the associated spectral sequences; that is a sequence of chain maps:
$$  \alpha_r:  (E_r,\delta_r) \longrightarrow (\overline{E}_r,\overline{\delta}_r), \ \  \text{satisfying}  \ \  (\alpha_r)_*=\alpha_{r+1} $$
\noindent  A standard treatment of this construction can be found in  \cite[pgs. 65-67]{McCleary}.   The perspective of reduction offers a  concrete construction of this morphism as follows.   To begin, note  that a filtered chain map decomposes into homogeneous summands in a manner similar to the differential $$ a=a^{(0)}+a^{(1)}+a^{(2)}+\cdots,\ \ \text{where} \ \ a^{(m)}: C^{(i)}\rightarrow \overline{C}^{(i+m)}. $$

\noindent Recalling that $E_0=C$ and $\overline{E}_0=\overline{C}$, let $a_0:E_0\rightarrow \overline{E}_0$ be defined as $a_0:=a$.   Let $$g_r:(E_r,d_r)\rightarrow (E_{r+1},d_{r+1}) \ \ \text{and} \ \  \overline{g}_r:(\overline{E}_r,\overline{d}_r)\rightarrow (\overline{E}_{r+1},\overline{d}_{r+1})$$ denote the chain homotopy equivalences that cancel the $r$-th order terms in the differential, and inductively define  $a_{r+1}: (E_{r+1},d_{r+1})\rightarrow (\overline{E}_{r+1},\overline{d}_{r+1})$ by $$a_{r+1}=\overline{g}_r\circ a_r\circ (g_r)^{-1},$$ where $(g_r)^{-1}$ is the  chain homotopy inverse of $g_r$.  Now  $a_r$ decomposes into summands according to the filtration, and it is easy to see that each $a_r$ respects the filtration, i.e.  
$$  a_r=  a_r^{(0)}+ a_r^{(1)}+ a_r^{(2)}+\cdots.$$  Indeed, this holds by assumption for $a_0=a$, and the homotopy equivalences $(g_r)^{-1}$ and $\overline{g}_r$ provided by the cancellation lemma are all non-decreasing in filtration degree (they involve terms of the form $h\circ d_r$ and $\overline{d}_r\circ \overline{h}$,  respectively, where $h,\overline{h}$  have degree $-r$ and $d_r,\overline{d}_r$ have degree $r$.)  Let the $0$--th order term of $a_r$ be denoted by:
$$ \alpha_r:= a_r^{(0)}.$$
The $0$--th order terms of the equality $a_r\circ d_r=\overline{d}_r\circ a_r$ must themselves be equal, which implies  $\alpha_r \circ \delta_r= \overline{\delta}_r\circ \alpha_r$.   Thus $\alpha_r$ is a chain map between $(E_r,\delta_r)$ and $(\overline{E}_r,\overline{\delta}_r)$, and by the definition of the map induced on homology, we have $(\alpha_r)_*=\alpha_{r+1}$.  This provides the desired morphism of spectral sequences.  

It should be pointed out that  one can verify that the morphism constructed above agrees with the more traditional construction.  Since we will not make use of this fact we again omit the proof.   We should also point out that in many situations, one or all of the intermediate chain maps $a_r$  may have vanishing $0$--th order terms; that is $a_r^{(0)}=0$. The same rational which shows  $a_r^{(0)}=\alpha_r$ is a chain map shows that the lowest order non-vanishing term is a chain map.  This allows one to  construct a morphism of spectral sequences where the filtration  shift of the maps between successive pages is monotonically increasing with the page index.

\subsection{An action on the link surgeries spectral sequence}\label{subsec:lsss}

Let  $L=K_1\cup \cdots \cup K_l \subset Y$ be a framed link.  In this subsection we prove that a curve $\gamma\subset Y\setminus L$ gives rise to a filtered endomorphism of the filtered complex that produces the link surgeries spectral sequence.  To state the theorem, we establish a bit of notation.  A {\em multi-framing} is an $l-$tuple  $I=\{m_1,...,m_l\}\in \hat{\Z}^l$ where $\hat{\Z}=\Z\cup\infty$.  A multi-framing specifies a $3$-manifold, which we denote $Y(I)$, by performing  surgery on $L$ with slope on the $i$-th component  given by $m_i$ (defined with respect to the base framing). Thus $Y(\infty,...,\infty)=Y$ corresponds to not doing surgery at all.  

\begin{thm}\label{thm:ssnaturality}
Let  $L\subset Y$ be a framed link.  There is a  filtered complex $\subcube$ whose homology is isomorphic to $\HFa(Y)$. The $E_1$ page of the associated spectral sequence satisfies $$E_1\cong \bigoplus_{I\in \{0,1\}^{|L|}} \HFa(Y(I)).$$
  A curve $\zeta \subset Y\setminus L$ induces a filtered chain map
$$  \curve^\zeta\co C(L)\longrightarrow C(L),
$$
whose induced map is isomorphic to the action $$ (\curve^{\zeta})_*=\Curve^
{[\zeta]}\co \HFa(Y)\longrightarrow \HFa(Y).$$
 The induced map on the $E_1$ page of the spectral sequence  
is given by the sum 
$$
\begin{CD}
\underset{I}\bigoplus \ \HFa(Y(I)) @>\underset{I}\oplus\Curve_I^{[\zeta]} >> \underset{I}\bigoplus\ \HFa(Y(I)),
\end{CD}
$$
where $\Curve_I^{[\zeta]}$ is the map on $\HFa(Y(I))$ induced by $\zeta$, viewed as a curve in  $Y(I)$.
\end{thm}

\begin{proof} The existence of the filtered complex computing  $\HFa(Y)$ is \cite[Theorem 4.1]{OSzBrCov}.  The refinement here is the existence of a filtered chain map associated to a curve in the  framed link diagram.   This map will be defined as a sum of the holomorphic polygon operators $\curve^\zeta_{0,1,...,n}$ from Subsection \ref{subsec:polygons}.  In order to make this precise, we  must recall the proof of \ons's theorem. There are essentially two main steps.  The first is to  construct a filtered chain complex $(X,D)$ from a framed link $L\subset Y$, and the second is to show that a natural (filtered) subcomplex $\subcube$ has homology isomorphic to $\HFa(Y)$.  In both steps, we will take care to introduce our refinement at the appropriate time.

The first step in the proof is to associate a  filtered complex to a Heegaard multi-diagram adapted to a framed link  $L=K_1\cup ... \cup K_l \subset Y$.   More precisely, we have a Heegaard multi-diagram $$(\Sigma, \etab^0,\etab^1,...,\etab^{k},w)$$ with  the property that the $3$--manifolds $Y_{0,i}$  ($i=1,...,k$) are in one-to-one correspondence with the $3$--manifolds $Y(I)$ associated to all possible multi-framings $I=\{m_1,...,m_l\}$, with $m_i\in \{0,1,\infty\}$.  The $3$--manifolds $Y_{i,j}$ for $i,j>0$ are all diffeomorphic to the connected sum of a number of  $S^1\times S^2$'s.  As usual, we require the multi-diagram to be  admissible.

 From such a Heegaard diagram one constructs a filtered complex $(X,D)$.  As a group, the complex splits as a direct sum over the set of  multiframings:$$X= \bigoplus_{I\in \{0,1,\infty\}^l}  \CFa(Y(I)).$$
The differential on X is given as a sum of holomorphic polygon maps associated to certain sub multi-diagrams.  Call a multi-framing $J$  an {\em immediate successor} of a multi-framing $I$ if the two framings differ on exactly one component  $K_i\subset L$ and if the restriction of $J$ to $K_i$ is one greater than the restriction of $I$, taken with respect to the length two ordering $0<1<\infty$.  Write $I<J$ if $J$ is an immediate successor of $I$.   Given any sequence of immediate successors $I^0<I^1<\dots<I^m$, there is a  map $$D_{I^0<\cdots<I^m}\co \CFa(Y(I^0))\longrightarrow \CFa(Y(I^m))$$ defined by the pseudo-holomorphic polygon map of Subsection~\ref{subsec:polygons}
$$ D_{I^0<\dots<I^m}(-):=f_{0,i_0,\dots,i_m}(-\otimes\Theta_1\otimes\dots\otimes\Theta_{m}).$$
The indexing   $(0,i_0,\dots,i_m)$ which specifies the sub multi-diagram $$(\Sigma,\etab^0,\etab^{i_0},\dots,\etab^{i_m},w)$$  is determined by the requirement that $Y_{0,i_j}= Y(I^j)$ for each $j=0,1,\dots,m$,  under the  bijection between the $3$-manifolds  $Y_{0,i}$ and those associated to the multi-framings. Here and throughout, $\Theta_j$ is a cycle generating the highest graded component of  $\HFa_*(Y_{i_{j-1},i_{j}})$.  For a sequence of successors of length zero, $D_I\co \CFa(Y(I))\rightarrow \CFa(Y(I))$ is simply the Floer differential.

\ons \ define an endomorphism  $D\co X\rightarrow X$ by $$D= \sum_{\{I^0<I^1<\cdots<I^m\}} D_{I^0<I^1<\cdots<I^m},$$
where the sum is over all sequences of immediate successors (of all lengths).  They show that $D\circ D=0$  \cite[Proposition 4.4]{OSzBrCov}.  The key tool  in their proof is the $A_\infty$ relation for pseudo-holomorphic polygon maps, together with a vanishing theorem for these maps in the present context, \cite[Lemma 4.2]{OSzBrCov}.

The resulting complex $(X,D)$ has a natural filtration by the totally ordered set $\{0,1,\infty\}^l$  which arises from its defining decomposition along multiframings.  We can also collapse this to a $\Z$-filtration, and it is this latter filtration which will be of primary interest.  To do this,  let the norm of a multi-framing $I=(m_1,...,m_l)$ be given by  $$|I|= \sum_{i=1}^{l} m'_i,$$ where $m'_i=m_i$ if $m_i\in \{0,1\}$ and $m'_i=2$ if $m_i=\infty$. Then $$X= \bigoplus_{i\ge0 } \ C^{(i)},$$
where $$C^{(i)}=\bigoplus_{\{I\in \{0,1,\infty\}^l  | \  |I|=i\}} \CFa(Y(I)).$$
The differential $D$ clearly respects the decomposition, giving rise to a decreasing filtration of $X$ by sub complexes  $F^p=\oplus_{i\ge p} C^{(i)}$, as in the previous subsection.   The $0$--th page of the spectral sequence is given by $$E_0= \bigoplus_{I\subset \{0,1,\infty\}^l} \CFa(Y(I)),$$ with $\delta_0$ differential given simply as the Floer differential; that is, $\delta_0|_{\CFa(Y(I))}=D_I$.  The $E_1$ page of the spectral sequence splits as $$E_1\cong  \bigoplus_{I} \HFa(Y(I)),$$
with  $\delta_1|_{\HFa(Y(I))}$  given by  $\sum (D_{I<J})_*$, the sum over all immediate successors of $I$, of the maps induced on homology by $D_{I<J}$.  These are   the maps $$\widehat{F}_W\co\HFa(Y(I))\rightarrow \HFa(Y(J))$$ associated to the $2$--handle cobordism between $Y(I)$ and $Y(J)$ (here we are using the fact that $J$ is an immediate successor of $I$).

We are now in a position to understand the action of a curve $\zeta \subset Y\setminus L$.  Without loss of generality, we may assume that $\zeta$ lies on the Heegaard surface, and is in general position with all attaching curves.  Define maps
$$\curve^\zeta_{I^0<\cdots<I^m}\co \CFa(Y(I^0))\rightarrow \CFa(Y(I^m))$$ by the polygon operators from Subsection~\ref{subsec:polygons}, $$\curve^\zeta_{I^0<\cdots<I^m}(-):=\curve^\zeta_{0,i_0,...,i_m}(-\otimes\Theta_1\otimes\cdots\otimes\Theta_{m}).$$ 
The sum of all these maps (over all sequences of immediate successors) is an endomorphism $\curve^\zeta\co X\rightarrow X$, which the following lemma shows is a chain map.  This map  forms the basis of our refinement.

\begin{lem}\label{lem:AChain}
The map $\curve^\zeta: X\rightarrow X$ given by $$\curve^\zeta= \sum_{\{I^0<I^1<\dots<I^m\}} \curve^\zeta_{I^0<I^1<\dots<I^m}$$
is a  chain map.  Moreover, $\curve^\zeta$  respects the collapsed filtration $F^i$ and the  filtration of $(X,D)$ by the totally ordered set $\{0,1,\infty\}^l$.
\end{lem}
\begin{proof}
We want to prove that $D\circ \curve^{\zeta}+\curve^{\zeta}\circ D=0$. The proof is analogous to \cite[Proposition~4.4]{OSzBrCov}. Given $I,J$ such that $|I|<|J|$, we expand the component map
$$D\circ \curve^{\zeta}+\curve^{\zeta}\circ D\co \widehat{CF}(Y(I))\to\widehat{CF}(Y(J))$$
to get 
\begin{equation} \label{eq:achainmap1}
\sum_{I=I^0<\cdots<I^m=J}\sum_{0\le r\le m}(D_{I^{r}<\cdots<I^m}\circ \curve^{\zeta}_{I^0<\cdots<I^r}+\curve^{\zeta}_{I^{r}<\cdots<I^m}\circ D_{I^0<\cdots<I^r}),
\end{equation}
where the first sum is over all sequences of immediate successors connecting $I$ to $J$.  Such sequences must be of length $m=|J|-|I|$.  Pick one of these sequences $I^0<\cdots< I^m$, and consider the term on the left  in the second summation, applied to a chain $\x$
$$\sum_{0\le r\le m}D_{I^{r}<\cdots<I^m}\circ \curve^{\zeta}_{I^0<\cdots<I^r}(\x):=\quad \quad \quad \quad \quad \quad\quad \quad \quad\quad \quad \quad \quad \quad \quad\quad \quad \quad$$
$$\quad \quad \quad\sum_{0\le r\le m} f_{0,i_r,\dots,i_m}(\curve^\zeta_{0,i_0,\dots,i_r}(\x\otimes \Theta_1\otimes \dots \otimes \Theta_{r})\otimes \Theta_{r+1}\otimes \dots \otimes \Theta_{m}),$$ where the indices $i_n$ refer to the sets of attaching curves for which $Y_{0,i_n}\simeq Y(I^n)$, as above.  Theorem~\ref{thm:AinfinityH1} indicates that this  is equal to
\begin{eqnarray*}
&&\sum_{0\le r<s\le m} \curve^\zeta_{0,i_0,\dots,i_r,i_s,\dots,i_m} (\x\otimes \dots\otimes f_{i_r,\dots,i_s}(\Theta_{r+1}\otimes ... \otimes \Theta_{s})\otimes...\otimes \Theta_{m})\\
&&+\sum_{0\le s\le m} \curve^\zeta_{0,i_s,\dots,i_m}( f_{0,i_0,\dots,i_s}(\x\otimes \dots \otimes \Theta_{s})\otimes\Theta_{s+1}\otimes\dots \otimes \Theta_{m}).
\end{eqnarray*}
But the second term in the above sum is  equal to, and hence cancels over $\F$, the   $\curve^\zeta\circ D$ term
  from the inner sum in \eqref{eq:achainmap1}. Thus   \eqref{eq:achainmap1}  becomes
 $$\sum_{I^0<\dots <I^m}\sum_{0\le r<s\le m} \curve^\zeta_{0,i_0,\dots,i_r,i_s,\dots,i_m} (\x\otimes \dots\otimes f_{i_r,\dots,i_s}(\Theta_{r+1}\otimes ... \otimes \Theta_{s})\otimes...\otimes \Theta_{m}),$$

This expression can be rewritten as
$$
\sum_{\begin{array}{c}
\scriptstyle
I'=:I^r\\
\scriptstyle J'=:I^s
\end{array}
} 
\!\!\!\!\! \sum_{
\begin{array}{c}
\scriptstyle
I^0<\cdots<I^r\\
\scriptstyle I^s<\cdots<I^m
\end{array}
}
\!\! \!\!\!\!\!\!\! \!\! \curve^{\zeta}_{0,i_0,..,i_r,i_s,..,i_m}
\big(\x\otimes..\otimes \!\!\!\!\! \!\!\!\!\! \sum_{
\begin{array}{c}
\scriptstyle
I^r<\cdots<I^s\\
\scriptstyle 
\end{array}
}
\!\!\!\!\! \!\!\!\!\!  f_{i_r,..,i_s}(\Theta_{r+1}\otimes..\otimes\Theta_{s})\otimes..\otimes\Theta_{m}\big)\\
$$
where $I'$ and $J'$ in the first sum range over all  multi-framings between $I$ and $J$ (which we subsequently relabel $I'=I^r$ and $J'=I^s$).  The second sum ranges over all sequences of successors connecting $I$ to $I'$ and $J'$ to $J$. The third sum ranges over all sequences of successors connecting $I'$ to $J'$.  However,  \cite[Lemma~4.3]{OSzBrCov}  states that $\sum_{I^r<\dots< I^s} f_{i_r,..,i_s}(\Theta_{r+1}\otimes..\otimes\Theta_{s})\equiv 0,$ for any $I^r$ and $I^s$.  This shows that $\curve^\zeta$ is a chain map.  It clearly respects the filtration $F^i$ (since it splits into homogeneous summands by the length of the sequence of successors) and the subcomplexes of the   filtration of $(X,D)$ by the totally ordered set $\{0,1,\infty\}^l$.
\end{proof}

The second step in the proof of \os's theorem is to show that $(X,D)$ has a natural quotient complex $\subcube$ whose homology is isomorphic to $\HFa(Y)$.  Given a subset of multiframings $S\subset \{0,1,\infty\}^l$,  let $X(S)$ denote the subgroup of $X$ given by $$X(S)=\bigoplus_{I\in S}\ \CFa(Y(I)).$$ Let $\subcube=X(\{0,1\}^l)$ be the group generated by all multiframings which do not contain $\infty$.  This is clearly a quotient complex, with associated subcomplex generated by those multiframings which contain $\infty$.  The complex $X(\{0,1\}^l)$ inherits the  filtration from $X$, and  the curve map $\curve^\zeta$  induces a filtered  chain map from  $X(\{0,1\}^l)$ to itself.   \ons \ show that the homology of $\subcube$ is isomorphic to $\HFa(Y)$, and our task is to show that their isomorphism fits into the following diagram 
$$
\begin{CD}
 \HFa(Y) @>\cong >>  H_*\subcube\\
@VV\Curve^{[\zeta]}V @VV(\curve^{\zeta})_*V\\
  \HFa(Y)@>\cong >>H_*\subcube .
\end{CD} 
$$

The key ingredient is a refined version of the strong form of the surgery exact triangle, \cite[Theorem 4.5]{OSzBrCov}.     Let $K$ be a framed knot in a $3$--manifold $Y$, and let $$f=D_{0<1}\co\CFa(Y(0))\longrightarrow \CFa(Y(1))$$ be the map induced by the $2$--handle cobordism.  Ozsv\'ath--Szab\'o \cite[Theorem 4.5]{OSzBrCov}  show that the mapping cone complex $M(f)$ is quasi-isomorphic to $\CFa(Y)$. We must account for the $\HFmod$ action.  To do this let $\zeta\subset Y\setminus K$ be a curve, as usual.   Consider the complex $$X(\{0,1,\infty\})=\CFa(Y(0))\oplus \CFa(Y(1))\oplus \CFa(Y(\infty)),$$ endowed with the differential \
$$D=\left( \begin{array}{ccc}
D_0 & 0 & 0 \\
D_{0<1} & D_1 & 0 \\
D_{0<1<\infty} & D_{1<\infty} & D_\infty \end{array} \right). $$

\noindent There is a natural short exact sequence
\begin{equation}\label{eq:mappingses} 0\rightarrow X(\{\infty\})\rightarrow X(\{0,1,\infty\})\rightarrow X(\{0,1\})\rightarrow0 ,\end{equation}
where the sub and quotient complexes are identified with $\CFa(Y)$ and $M(f)$, respectively.  \ons \ show that $X(\{0,1,\infty\})$ is acyclic \cite[Lemma~4.2]{OSzBrCov},  proving that $M(f)$ is quasi-isomorphic to $\CFa(Y)$.

We have the map
$$  \curve^{\zeta}: X(\{0,1,\infty\})\longrightarrow X(\{0,1,\infty\})$$ 
given by 
$$\curve^\zeta=\left( \begin{array}{ccc}
\curve^\zeta_0 & 0 & 0 \\
\curve^\zeta_{0<1} & \curve^\zeta_1 & 0 \\
\curve^\zeta_{0<1<\infty} & \curve^\zeta_{1<\infty} & \curve^\zeta_\infty \end{array} \right),$$ which Lemma~\ref{lem:AChain} shows is a chain map.  Moreover, $\curve^\zeta$  respects the short exact sequence \eqref{eq:mappingses}, thus inducing a map on the sub and quotient complex.   The map on the subcomplex is simply $\curve^\zeta_\infty\co\CFa(Y_\infty)\rightarrow\CFa(Y_\infty)$ and the map on the quotient complex is 
$$ \curve^\zeta_{M(f)} = \begin{pmatrix} \curve^\zeta_0 & 0 \\ \curve^\zeta_{0<1} & \curve^\zeta_1 \end{pmatrix}\co M(f)\rightarrow M(f).$$
Considering the corresponding long exact sequence in homology, we obtain a commutative diagram:
$$
\begin{CD}
 ... @>>> H_*(M(f)) @>>> H_*(X(\{0,1,\infty\})@>>>  \HFa_*(Y_\infty) @>>>...\\
@. @VV{(\curve^\zeta_{M(f)})_*}V @VV{(\curve^\zeta)_*}V  @VV{(\curve^\zeta_\infty)_*}V @. \\
 ... @>>> H_*(M(f)) @>>> H_*(X(\{0,1,\infty\})@>>>  \HFa_*(Y_\infty) @>>>...,
\end{CD} 
$$
By \ons's theorem, $H_*(X(\{0,1,\infty\})=0$, showing that 
$$
\begin{CD}
 \HFa_{*+1}(Y)@>\cong >> H_*(M(f)) \\
 @VV{(\curve^{\zeta})_{*+1}}V  @VV{(\curve^\zeta_{M(f)})_*}V \\
  \HFa_{*+1}(Y)@>\cong >> H_*(M(f)). 
\end{CD} 
$$ \noindent This can be interpreted as saying that the Floer homology of $Y$ is isomorphic to the mapping cone of the $2$--handle map, {\em as a module} over $\HFmod$.

Given this refinement of the surgery exact triangle, the proof of  Theorem~\ref{thm:ssnaturality} proceeds quickly by the same inductive argument used in the proof of \cite[Theorem 4.1]{OSzBrCov}.  Specifically, we return to the  complex $X$ coming from a framed link diagram of $l$ components.  If $l=1$, the preceding discussion proves the theorem, showing that the filtered complex $X(\{0,1\})$ computes the homology of $X(\infty)=\CFa(Y)$, and that a curve $\zeta\subset Y\setminus K$ induces a filtered chain map whose induced map agrees with that of $\curve^{\zeta}\co\CFa(Y)\rightarrow \CFa(Y)$.  Assume that this remains true for the complex $X(\{0,1\}^{l-1})$ associated to an $(l-1)$--component link and the induced map $\curve^\zeta$.  That is, we have a commutative diagram: 
$$\begin{CD}
 \HFa_{*+1}(Y)@>\cong >> H_*(X(\{0,1\}^{l-1})) \\
 @VV{(\curve^{\zeta})_{*+1}}V  @VV{(\curve^\zeta)_*}V \\
  \HFa_{*+1}(Y)@>\cong >> H_*(X(\{0,1\}^{l-1})). 
\end{CD} 
$$

Turn now to an $l$--component link.  In this case, we consider the complex $X(\{0,1\}^{l-1}\times \{0,1,\infty\})$ (that this is a complex follows from the fact that it is a quotient  of $X$ by the subcomplex  consisting of multiframings with at least one of the first $l-1$ parameters equals to $\infty$) .  There is the short exact sequence, compatible with the maps induced by $\curve^\zeta$:
$$\minCDarrowwidth{15pt} 
\begin{CD}
 0 @>>> X(\{0,1\}^{l-1}\times\{\infty\})@>>>X(\{0,1\}^{l-1}\times \{0,1,\infty\})@>>>  X(\{0,1\}^l)  @>>>0\\
@. @VV{\curve^\zeta}V @VV{\curve^\zeta}V  @VV{\curve^\zeta}V @. \\
0 @>>> X(\{0,1\}^{l-1}\times\{\infty\})@>>>X(\{0,1\}^{l-1}\times \{0,1,\infty\})@>>>  X(\{0,1\}^l)  @>>>0.
\end{CD}
$$
The middle term has a natural filtration coming from the total ordering on $\{0,1\}^{l-1}$.  The associated graded groups of this filtration are each isomorphic to $H_*(X(I\times \{0,1,\infty\}))$ for some multiframing $I\subset \{0,1\}^{l-1}$.  The strong form of the surgery exact triangle, however, implies that these groups are all zero. It follows that $H_*(X(\{0,1\}^{l-1}\times \{0,1,\infty\}))=0$.  Thus we have the diagram 

$$
\begin{CD}
  H_{*+1}(X(\{0,1\}^{l-1}\times\{\infty\}))@>\cong >> H_*(X(\{0,1\}^{l})) \\
 @VV{(\curve^{\zeta})_{*+1}}V  @VV{(\curve^\zeta)_*}V \\
  H_{*+1}(X(\{0,1\}^{l-1}\times\{\infty\}))@>\cong >> H_*(X(\{0,1\}^{l})). 
\end{CD} 
$$
Our inductive hypothesis equates the left hand side with $\Curve^{[\zeta]}\co\HFa(Y)\rightarrow \HFa(Y)$. This completes the proof that  the map induced on homology by $\curve^\zeta$ is isomorphic to $A^{[\zeta]}$.

To see that the  induced  map  $\alpha_1^{\zeta}$  on the $E_1$ page  is given by 
$$
\begin{CD}
\underset{I}\bigoplus \ \HFa(Y(I)) @>\underset{I}\oplus\Curve_I^{[\zeta]} >> \underset{I}\bigoplus\ \HFa(Y(I)),
\end{CD}
$$
it suffices to recall that   $\alpha_1^{\zeta}$ is the induced map on the  homology of $(E_0,\delta_0)$ by $\alpha_0^{\zeta}$, where $\alpha^\zeta_0$ is the map on the $E_0$ page induced by $\curve^{\zeta}$.   But $\alpha^\zeta_0$, in turn, is simply the lowest order term of $\curve^{\zeta}$ and is given by 
$$
\begin{CD}
\underset{I}\bigoplus \ \CFa(Y(I)) @>\underset{I}\oplus\curve_I^{\zeta} >> \underset{I}\bigoplus\ \CFa(Y(I)),
\end{CD}
$$
where $\curve_I^\zeta:\CFa(Y(I))\rightarrow \CFa(Y(I))$ is  the  operator obtained by viewing $\zeta$ as a curve in $Y(I))$.   By definition, we have $A_I^{[\zeta]}=(\curve_I^\zeta)_*$. 
\end{proof}

\subsection{Connecting the Khovanov module to the Floer module}\label{subsec:khss}

As discussed in Section \ref{sec:Kh}, a marked point $p_0$ on one component of a link diagram gives rise to the reduced Khovanov chain complex, $\CKr$.  Choosing a marked point  on each remaining component gives the reduced Khovanov homology an $\modred$--module structure.    More precisely, additional marked points give rise to chain maps $$x_i\co \CKr(L)\rightarrow \CKr(L), \ \ i=1,...,n-1$$ which satisfy $x_i\circ x_j=x_j\circ x_i$ and $x_i\circ x_i=0$.  

Consider a properly embedded arc $t_i\subset (S^3,L)$ connecting  $p_0$ to the  $i$--th additional marked point.  The preimage of this arc in the branched double cover is a closed curve $\zeta_i= \pi^{-1}(t_i)\subset \Sigma(L)$.  As in Subsection \ref{subsec:H_1},  we can assume that $\zeta_i$ lies on the Heegaard surface of a Heegaard diagram of $\Sigma(L)$, and hence we obtain a chain map on the associated Floer complex:
$$  \curve^{\zeta_i}\co  \CFa(\Sigma(L))\longrightarrow \CFa(\Sigma(L)).$$

\noindent This chain map is related to the chain map $x_i$ on the Khovanov complex by the following theorem.

\begin{thm}\label{thm:Khnaturality} 
Let $D$ be a diagram for a link $L=K_0\cup...\cup K_{n-1}\subset S^3$, together with a base point $p_0\subset K_0$.   There is a filtered chain complex, $(C(D),d)$, whose homology is isomorphic to $\widehat{HF}(\Sigma(L))$.  The associated spectral sequence satisfies $$(E_1, \delta_1)\cong (\CKr(\overline{D}),\partial)$$
where $\overline{D}$ is the mirror of $D$ and the reduced Khovanov complex is defined with  $p_0$.   Let  $t_i$ be a proper arc connecting $p_0\subset K_0$ to $p_i\subset K_i$, and $\zeta_i= \pi^{-1}(t_i)\subset \Sigma(L)$ its lift to the branched double cover.  Then there is a filtered chain map  
$$  \curve^{\zeta_i}\co   (C(D),d)\longrightarrow (C(D),d),$$ whose induced map on homology  satisfies $$(\curve^{\zeta_i})_*=\Curve^{[\zeta_i]}\co\HFa(\Sigma(L))\rightarrow \HFa(\Sigma(L)).$$ The induced map $\alpha_1^{\zeta_i}$ on the $E_1$ page of the spectral sequence satisfies

$$\begin{CD}
 (E_1,\delta_1) @>\alpha_1^{\zeta_i} >> (E_1,\delta_1) \\
 @VV{\cong}V  @VV{\cong}V \\
(\CKr(\overline{D}),\partial)@>x_i >> (\CKr(\overline{D}),\partial). 
\end{CD} 
$$

 \end{thm}
 \begin{proof}
 The first part of the theorem is the content of \cite[Theorem 1.1]{OSzBrCov}, which is a rather immediate consequence of the link surgeries spectral sequence, applied to a particular surgery presentation of $\Sigma(L)$ coming from the link diagram.  The key point is that the branched double covers of links which differ by the  unoriented skein relation differ by a triad of surgeries along a framed knot in that manifold.  Moreover, the branched double cover of the natural saddle cobordism passing between the zero and one resolution is the $2$--handle cobordism  between the branched double covers which appears in the surgery exact triangle. The branched double cover of a complete resolution is diffeomorphic to $\#^{k}S^1\times S^2$, where $k+1$ is the number of components of the resolution, and the Heegaard Floer homology of $\#^{k}S^1\times S^2$ is isomorphic to the reduced Khovanov homology of the complete resolution.  Moreover, the maps  between the  Floer homology of the connected sums of $S^1\times S^2$'s induced by the  $2$--handle cobordisms  agree with the Frobenius algebra  defining the reduced Khovanov differential.
 
  In the present situation, we wish to keep track of an action. In the case of Khovanov homology it is the action $x_i$ induced by a point $p_i\in K_i$, whereas in the Floer setting it is the action of the curve $\zeta_i=\pi^{-1}(t_i)$ arising as the lift of an arc $t_i$ connecting $p_0$ to $p_i$.  Theorem~\ref{thm:ssnaturality}        shows that $\zeta_i$ induces a filtered chain map $$\curve^{\zeta_i}\co X(\{0,1\}^l)\rightarrow X(\{0,1\}^l),$$ whose induced map on homology agrees with $\Curve^{[\zeta_i]}\co \HFa(\Sigma(L))\rightarrow \HFa(\Sigma(L))$.  Thus it suffices to see that the induced map on the $E_1$ page of the spectral sequence agrees with the Khovanov action.   This follows from the fact that the induced map $\alpha_1^{\zeta_i}: (E_1,\delta_1)\rightarrow (E_1,\delta_1)$   is simply  the sum of the actions of $\zeta_i$ on the Floer homology groups at each individual vertex in the cube; that is,

  $$\begin{CD}
 E_1 @>\alpha_1^{\zeta_i} >> E_1 \\
 @VV{\cong}V  @VV{\cong}V \\
\underset{I\subset \{0,1\}^l}\bigoplus \HFa(Y(I))@>\underset{I}\oplus A_I^{[\zeta_i]} >> \underset{I\subset \{0,1\}^l}\bigoplus \HFa(Y(I)),
\end{CD} 
$$
where $A_I^{[\zeta_i]}: \HFa(Y(I))\rightarrow \HFa(Y(I))$ is the map induced by $\zeta_i$, viewed as  curve in $Y(I)$.  Now each $Y(I)$ is diffeomorphic to $\#^k S^1\times S^2$ for some $k$, and  $$\HFa(\#^k S^1\times S^2) \cong \F[x_1,...,x_{k}]/(x_1^2,...,x_k^2)$$ as a module over $\Lambda^*(H_1(\#^{k}S^1\times S^2))$.   Here  a basis for $H_1(\#^{k} S^1\times S^2)$ is identified with the variables $x_1,...,x_{k}$, and such a basis is given by the lift of any $k$ proper arcs connecting the unknot  containing $p_0$ to each of the $k$ other unknots.  Under this correspondence, the action of $\zeta_i$ on the Floer homology of the branched double cover of a complete resolution agrees with the action of  $x_i$ on the reduced Khovanov homology of the complete resolution.  This completes the proof.
\end{proof}

The preceding theorem allows us to endow the entire \os \ spectral sequence with a  module structure.  To describe this, we say that a ring $R$ {\em acts on a spectral sequence} $\{(E_i,\delta_i)\}_{i=0}^\infty$ if for each $i$ we have
\begin{enumerate} 
\item  $E_i$ is an $R$--module.
\item The differential is $R$--linear:   $x\cm \delta_i(\beta)=\delta_i(x\cm \beta)$ for all $x\in R, \beta\in E_i$.  Equivalently, $R$ acts on $E_i$ by chain maps.
\item The $R$--module structure on $E_{i+1}$ is induced through homology by the module structure on $E_i$. \end{enumerate}
If the above conditions hold only for $i\ge r$, we say that the action {\em begins} at the $r$--th page. The following is now an easy corollary of Theorem \ref{thm:Khnaturality}.

\begin{cor}\label{E_2} Let $L=K_1\cup...\cup K_n\subset S^3$ be a link. Then  the spectral sequence from the reduced Khovanov homology of $\overline{L}$ to the Floer homology of $\Sigma(L)$ is acted on by $\modred$.  The resulting module structure on Khovanov homology is isomorphic to the module structure induced by the basepoint maps.  
\end{cor}
\begin{rem} Reduced  Khovanov {\em homology} appears as the $E_2$ page of a spectral sequence.   As our proof  indicates, the action begins at  $E_1$.  The action on $E_0$ holds only up to chain homotopy.
\end{rem}
\begin{proof}  For each component  $K_i\subset L$, choose a  basepoint $p_i\subset K_i$ on a diagram for $L$.  Pick a system of arcs $t_i$ connecting $p_0$ to $p_i$, and consider the closed curves $\zeta_i=\pi^{-1}(t_i), i=1,...,n-1,$ arising from their lifts to $\Sigma(L)$.   According to Theorem \ref{thm:Khnaturality} we obtain a collection of filtered chain maps,
$$ a^{\zeta_i}:  (C,d)\rightarrow (C,d),  \ \ i=1,...,n-1,$$
acting on the filtered chain complex which gives rise to the \os \ spectral sequence.

 Consider the free group $F_{n-1}$ on $n-1$ generators, $X_1,...,X_{n-1}$.  There is an obvious action of the group algebra $\F[F_{n-1}]$ on the spectral sequence: simply define the action of $X_i$ on  $E_r$ to be $\alpha_r^{\zeta_i}$, the map induced by $\curve^{\zeta_i}$ on $E_r$ and extend this to $\F[F_{n-1}]$ in the natural way.  Thus an element such as $X_1X_2 + X_6$ acts on $E_r$ by the chain map $\alpha_r^{\zeta_1}\circ\alpha_r^{\zeta_2}+ \alpha_r^{\zeta_6}$.  The fact that for each curve $\zeta_i$ and each $r$, $$\alpha_r^{\zeta_i}: (E_r,\delta_r)\rightarrow (E_r,\delta_r),$$ is a chain map satisfying,   $\alpha_{r+1}^{\zeta_i}= (\alpha_{r}^{\zeta_i})_*$, implies that this is indeed an action by $\F[F_{n-1}]$.  To see that the action descends to $\modred$, it suffices to check that $X_iX_j+X_jX_i$ and $X_i^2$ act as zero on each page of the spectral sequence, for all  $i,j$.  These relations clearly hold on the $E_1$ page, however, since Theorem \ref{thm:Khnaturality}  implies that the  map induced by $\curve^{\zeta_i}$ on $E_1$ agrees with $x_i$, the map on the reduced Khovanov complex  induced by $p_i$.   But if an endomorphism of a spectral sequence  is zero on some page, it is zero for all subsequent pages, since $\alpha_{r+1}=(\alpha_{r})_*$.  It follows that $X_iX_j+X_jX_i$ acts as zero on $E_r$ for all $r\ge 1$. 
 
 To show that the module structure on $E_2$ agrees with the module structure on Khovanov homology induced by the basepoint maps it suffices to note, again, that $\alpha_1^{\zeta_i}=x_i$ and $\alpha_2^{\zeta_i}=(\alpha_1^{\zeta_i})_*=(x_i)_*$.
 \end{proof} 

It is natural to ask about convergence  of the $\modred$--action.  In particular, since the homology of the spectral sequence converges to the Floer homology of $\Sigma(L)$, one could hope that the action converges to the $\Lambda^*(H_1(Y;\mathbb F))$-action  (where the homology classes   $[\zeta_i]$ serve as a spanning set for $H_1(Y;\mathbb F)$.)  While this is true it is  not necessarily as useful as one may  think, since convergence is phrased in terms of the associated graded module (see \cite{McCleary}) and this module may have many extensions.  Put differently, the module structure on the associated graded is sensitive only to the lowest order terms of the filtered chain maps $\curve^{\zeta_i}$, and the higher order terms may contribute in a non-trivial way to the module structure on $\HFa(\Sigma(L))$.  Despite this, the corollary can still be used to prove the following ``collapse result" for the spectral sequence.  This theorem will be one of our main tools for showing that the Khovanov module detects unlinks.

\begin{prop}\label{collapse}
Suppose $$\Khr(\overline{L})\cong \modred$$ as a module over $\modred$.  Then $$\HFa(\Sigma(L))\cong \modred$$ as a module over $\Lambda^*(H_1(\Sigma(L);\mathbb F))$. 
\end{prop}
\begin{proof}
The assumption on the reduced Khovanov homology implies, in particular, that the entire Khovanov homology is supported in homological grading zero.  This is because $1\in \Khr(L)$ generates the homology as a module over $\modred$, and the module action preserves the homological grading (it is induced by chain maps of degree zero).  

In the proof of Theorem \ref{thm:Khnaturality}, however,  the filtration on the complex computing $\HFa(\Sigma(L))$ arises from a graded decomposition coming from the norm of a multi-framing 
$$ X(\{0,1\}^l)= \bigoplus_{i\ge 0} \ C^{(i)}, \ \ \ \  \ \  C^{(i)}= \bigoplus_{ \{I \in \{0,1\}^l | \ |I|=i\}} \ \CFa(Y(I))$$ and the norm $|I|$
 corresponds to the homological grading on the Khovanov complex.  Thus the $E_2$ page of the spectral sequence, which is identified with $\Khr(L)$, is supported in a single filtration. Since the higher differentials strictly lower the filtration, it follows that the spectral sequence has collapsed and that  $\Khr(\overline{L})\cong E_2=E_\infty\cong \HFa(\Sigma(L))$.  Moreover, since the module structure on $E_\infty$ is induced by $E_2$ through homology, it follows that this is an isomorphism of $\modred$ modules.  We claim that the latter module structure agrees with the  $\Lambda^*(H_1(\Sigma(L);\mathbb F))$--module structure on Floer homology.

To see this, note first that the curves $\zeta_i$ span $H_1(\Sigma(L);\F)$. Now on  one hand we have the filtered chain maps $\curve^{\zeta_i}$ whose induced maps on homology agree with $A^{[\zeta_i]}:\HFa(\Sigma(L))\rightarrow \HFa(\Sigma(L))$.  On the other we have $X_i$,  the maps induced by $\curve^{\zeta_i}$ on the $E_\infty$  page of the spectral sequence.  
To see that $A^{[\zeta]}=X_i$, it suffices to recall the construction from Section \ref{subsec:ssconstruction} of the morphism of spectral sequences induced by  $\curve^{\zeta_i}$.  By construction, each filtered  map $\curve^{\zeta_i}$ is chain homotopic to a map $$\curve_\infty^{\zeta_i}: E_\infty \rightarrow E_\infty, \ \ \ \curve_\infty^{\zeta_i}= \curve_\infty^{(0)}+ \curve_\infty^{(1)}+\curve_\infty^{(2)}+...,$$
and the morphism induced on the $E_\infty$ is, by definition, the lowest order term in this map:  $\alpha^{\zeta_i}_\infty:=\curve_\infty^{(0)}$.  But the discussion above indicates that $E_\infty$ is supported in a single filtration summand, hence the higher order terms in the decomposition of $\curve_\infty^{\zeta_i}$ vanish.  It follows that the maps  $\Curve^{[\zeta]}=(\curve^{\zeta_i})_*=\curve_\infty^{\zeta_i}$  and $X_i= a_\infty^{(0)}$ are equal.  
The proposition follows.\end{proof}

\section{A nontriviality theorem for homology actions}
In this section we prove a non-triviality result for Floer homology, 
Theorem~\ref{thm:ContNonzero}.  Roughly speaking, the theorem says that the homology of the Floer homology with respect to the action of any curve is non-trivial, provided a manifold does not contain a homologically essential $2$--sphere.  This  detection theorem for Floer homology  will transfer through the spectral sequence of the previous section to our detection theorem for Khovanov homology.

Following Kronheimer and Mrowka \cite{KMsuture}, let
$$HF^{\circ}(Y|R)=\bigoplus_{\{\mathfrak s\in\mathrm{Spin}^c(Y) | \langle c_1(\mathfrak s),[R]\rangle=x(R) \} }HF^{\circ}(Y,\mathfrak s),$$
where $R$ is a Thurston norm minimizing surface in $Y$ and $x(R)$
is its Thurston norm.

\begin{thm}\label{thm:ContNonzero}
Suppose $Y$ is a closed, oriented and irreducible $3$--manifold
with $b_1(Y)>0$. Let $R\subset Y$ be a Thurston norm minimizing
connected surface. Then there exists a cohomology class
$[\omega]\in H^2 ( Y;\mathbb Z )$ with $\langle
[\omega],[R]\rangle>0$, such that for any $[\zeta]\subset
H_1(Y;\mathbb Z)$ the homology group with respect to $\Curve^{[\zeta]}$
$$H(\underline{\widehat{HF}}(Y|R;\mathcal R_{[\omega]}),\Curve^{[\zeta]})$$
has positive rank as an $\mathcal R$--module.
\end{thm}
\begin{proof}
We adapt the argument of Ozsv\'ath and Szab\'o
\cite[Theorem~4.2]{OSzGenus}.

By Gabai \cite{G1}, there exists a taut foliation $\mathscr F$ of
$Y$, such that $R$ is a compact leaf of $\mathscr F$. The work of
Eliashberg and Thurston \cite{ET} shows that $\mathscr F$ can be
approximated by a weakly symplectically semi-fillable contact
structure $\xi$, where $Y\times[-1,1]$ is the weak semi-filling.

By Giroux \cite{Giroux}, $Y$ has an open book decomposition which supports $\xi$.  By plumbing positive Hopf bands to the page of such an open book, we may assume that the binding is connected and that the genus of the page is greater than one.   Let $K\subset Y$ denote the
binding, and let $Y_0$ be the fibered $3$--manifold obtained from $Y$ by $0$--surgery on $K$.  There is a  $2$--handle cobordism $W_0\co Y\to
Y_0$.   Similarly, there is a $2$--handle cobordism $-W_0\co Y_0\to Y$, obtained by reversing the orientation of $W_0$ and viewing it ``backwards". 

 Eliashberg \cite{El}  shows that 
the weak semi-filling $Y\times [-1,1]$ can be embedded in a closed symplectic $4$--manifold $X$ (see also Etnyre \cite{Et}, for an alternative construction).   This is done by constructing symplectic caps for the boundary components.  Eliashberg's caps are produced by first equipping the $2$-handle cobordisms $W_0$ and $-W_0$  with appropriate symplectic structures, and then extending these structures over  Lefschetz fibrations $V_0$ and $V_0'$ whose boundaries are the fibered $3$--manifolds $-Y_0$ and  $Y_0$, respectively (note  $V_0$ and $V_0'$ are not, in general, orientation-reversing diffeomorphic).     Moreover, \cite[Theorem~1.3]{El} says that we can
choose $V_0$ and $V_0'$ so that \begin{equation}\label{eq:H1=0}
H_1(V_0)=H_1(V_0')=0.
\end{equation}

The result of the construction is a closed symplectic $4$--manifold, $(X,\omega)$, which decomposes as
$$X=V_0'\underset{-Y_0}\cup-W_0\underset{-Y=Y\times\{-1\}}\cup Y\times[-1,1]\underset{-Y}\cup W_0\underset{-Y_0}\cup V_0.$$ We view the cobordism from right to left, so that the orientation shown on a $3$--manifold is that which it inherits as the oriented boundary of the $4$--manifold to the right of the union in the decomposition.  By perturbing $\omega$ slightly and  multiplying by an
integer, we may assume $[\omega]\in H^2(X;\mathbb Z)$. We can
also arrange that $b_2^+(V_0')>1$ and $b_2^+(V_0)>1$, and hence can decompose $V_0$  by an
admissible cut along a $3$--manifold,  $N$. Thus $X=X_1\cup_N X_2$ with $b_2^+(X_i)>0$,
and
$$X_2=V_0'\underset{-Y_0}\cup-W_0\underset{-Y}\cup Y\times[-1,1]\underset{-Y}\cup W_0\underset{-Y_0}\cup (V_0\setminus X_1). $$

Denote  the canonical Spin$^c$ structure associated  to $\omega$  by $\mathfrak k(\omega)$, and
  the restriction of $\mathfrak k(\omega)$ to
$Y_0$ by $\mathfrak t$. Let $c(\xi;[\omega])\in
\underline{\widehat{HF}}(-Y|R;\mathcal R_{[\omega]})/\mathcal R^\times$ be the
twisted Ozsv\'ath--Szab\'o contact invariant defined in
\cite{OSzCont,OSzGenus}, and let $c^+(\xi;[\omega])\in
\underline{HF}^+(-Y|R;\mathcal R_{[\omega]})/\mathcal R^\times$ be its image under
the natural map $\iota_*:\underline{\widehat{HF}}\to \underline{HF}^+$. Let $\pi\co Y_0\to S^1$ be
the fibration on the $0$--surgery induced by the open book decomposition of $Y$, and let
$c(\pi)$ be a generator of
$\underline{\widehat{HF}}(-Y_0,\mathfrak t;\mathcal R_{[\omega]})$
whose image in $\underline{HF}^+(-Y_0,\mathfrak t;\mathcal
R_{[\omega]})\cong\mathcal R$ is a generator $c^+(\pi;[\omega])$. By \cite[Proposition~3.1]{OSzCont}
\begin{equation}\label{eq:F_W0}
\underline{\widehat F}_{W_0;[\omega]}(c(\pi))\doteq  c(\xi;[\omega]),
\end{equation}
where ``$\doteq$" denotes equality, up to multiplication by a unit.
The following commutative diagram summarizes the relationship between the contact invariants:
$$
\begin{CD}
c(\pi) @>\underline{\widehat F}_{W_0;[\omega]}>>  c(\xi;[\omega])\\
@V\iota_*VV @V\iota_*VV\\
c^+(\pi;[\omega])@>\underline{F}^+_{W_0;[\omega]}>> c^+(\xi;[\omega]).
\end{CD}
$$

Let $W=V_0'\cup -W_0$, and let $B$ be an open $4$-ball in $V_0'$.  As $W$ is a symplectic filling of $(Y,\xi)$,   the argument in
\cite[proof of Theorem 4.2]{OSzGenus} shows that
$$\underline{F}^+_{W\setminus B;[\omega]}(c^+(\xi;[\omega]))$$ is a
non-torsion element in $\underline{HF}^+(S^3;\mathcal
R_{[\omega]})$. Note that here, as above, we regard $W$ from right to left; namely, as a cobordism from $-Y$ to $S^3$.  It follows that
\begin{equation}\label{eq:F_W}
\underline{\widehat{F}}_{W\setminus B;[\omega]}(c(\xi;[\omega]))\quad
\text{is non-torsion.}
\end{equation}

One can construct  a Heegaard diagram for $-Y_0$ such that there
are only two intersection points representing $\mathfrak t$, and such that
there are no holomorphic disks connecting them which avoid the hypersurface specified by the basepoint.  Indeed, such a Heegaard diagram is constructed in the course of the proof of  \cite[Theorem 1.1]{OSzCont}, and its desired properties are verified in the  proof of \cite[Proposition 3.1]{OSzCont} (here we use the fact that the page of the open book has genus greater than one, though \cite[Section 3, specifically Remark 3.3]{Wu} indicates that the same technique can be adapted for genus one open books).   Since $H_1(Y;\Z)$ is naturally a subgroup of $H_1(Y_0;\Z)$,  any $[\zeta]\in H_1(Y)$ can
be viewed as an element in $H_1(Y_0;\Z)$, and 
 the preceding discussion implies that $\Curve^{[\zeta]}=0$ on
$\underline{\widehat{HF}}(-Y_0,\mathfrak t;\mathcal
R_{[\omega]})$ for every $[\zeta]$. Using (\ref{eq:F_W0}) and Theorem~\ref{thm:Cob},
\begin{eqnarray*}
\Curve^{[\zeta]}(c(\xi;[\omega]))&\doteq&\Curve^{[\zeta]}\circ\underline{\widehat F}_{W_0;[\omega]}( c(\pi))\\
&=&\underline{\widehat F}_{W_0;[\omega]}\circ \Curve^{[\zeta]}( c(\pi))\\
&=&\underline{\widehat F}_{W_0;[\omega]}(0)=0.
\end{eqnarray*}
Hence $c(\xi;[\omega])\in\mathrm{ker}(\Curve^{[\zeta]})$.

Now if $kc(\xi;[\omega])\in\mathrm{im}(\Curve^{[\zeta]})$ for some
nonzero $k\in \mathcal R$, then there is an element $a\in
\underline{\widehat{HF}}(Y;\mathcal R_{[\omega]})$ such that
$\Curve^{[\zeta]}(a)=kc(\xi;[\omega])$. Using (\ref{eq:H1=0}) and
Theorem~\ref{thm:Cob},
\begin{eqnarray*}
\underline{\widehat F}_{W\setminus B;[\omega]}(kc(\xi;[\omega]))&=&\underline{\widehat F}_{W\setminus B;[\omega]}\circ \Curve^{[\zeta]}(a)\\
&=&\Curve^{0}\circ \underline{\widehat F}_{W\setminus B;[\omega]}(a)\\
&=&0,
\end{eqnarray*}
a contradiction to (\ref{eq:F_W}). Hence
$kc(\xi;[\omega])\notin\mathrm{im}(\Curve^{[\zeta]})$. It follows that
$c(\xi;[\omega])$ represents a non-torsion element in
$H(\underline{\widehat{HF}}(Y|R;\mathcal
R_{[\omega]}),\Curve^{[\zeta]})$, so our conclusion holds.
\end{proof}

\begin{cor}\label{cor:ConnNonzero}
Suppose $Y$ is a closed, oriented $3$--manifold which does not
contain $S^1\times S^2$ connected summands. Then there exists a
cohomology class $[\omega]\in H^2(Y;\mathbb Z)$ such that for any
$[\zeta]\subset H_1(Y;\mathbb Z)$ the homology group with respect
to $\Curve^{[\zeta]}$
$$H(\underline{\widehat{HF}}(Y;\mathcal R_{[\omega]}),\Curve^{[\zeta]})$$
has positive rank as an $\mathcal R$--module.
\end{cor}
\begin{proof}
In the case that $b_1(Y)>0$, this follows from Theorem~\ref{thm:ContNonzero} and the twisted
version of (\ref{eq:ConnSum}).  When $b_1(Y)=0$, the theorem holds with $[\omega]=0$.  Indeed, when $[\omega]=0$ the corresponding Floer homology group $\HFa(Y;\mathcal R_{[\omega]})$ is simply the homology of $\CFa(Y)\otimes_\F \F[T,T^{-1}]$; that is, we take the untwisted complex and tensor over $\F$ with $\F[T,T^{-1}]$.  Now \cite[Proposition 5.1]{OSzAnn2} indicates that the Euler characteristic of $\HFa(Y)$, and hence its rank over $\F$, is non-trivial.   The universal coefficient theorem then implies $H_*(\CFa(Y)\otimes_\F \F[T,T^{-1}])$ has positive rank as an $\F[T,T^{-1}]$--module.  Finally, since every class  $[\zeta]\in H_1(Y;\Z)$ is torsion, Lemma \ref{homologyaction} shows that the $\Curve^{[\zeta]}=0$, as an operator on $\HFa(Y;\mathcal R)$.
\end{proof}

\section{Links with the Khovanov module of an unlink}
We now bring together the results from previous sections to prove our main theorems.  The first task is to prove Theorem \ref{thm:HFmodUnique}, which states that  Heegaard Floer homology, as a module over $\Lambda(H^1(Y;\F))$, detects $S^1\times S^2$ summands in the prime decomposition of a closed oriented $3$--manifold.  By way of the module structure on the spectral sequence  from  Khovanov homology to Heegaard Floer homology (specifically Proposition \ref{collapse}), this detection theorem will quickly lead to the Khovanov module's detection of unlinks, Theorem \ref{thm:main}.

The detection theorem for the Heegaard Floer module  makes  use of 
 of Corollary~\ref{cor:ConnNonzero} from the previous section.  The main challenge is to take this corollary, which is a non-vanishing result for homology actions on Heegaard Floer homology with {\em twisted coefficients}, and use it to obtain a characterization result for the Floer homology module with untwisted, i.e. $\F$, coefficients.   
Not surprising, to pass from twisted coefficients to untwisted coefficients we will
need the universal coefficients theorem. Let us recall its
statement from Spanier \cite{Sp}.

\begin{thm}\label{thm:UCT}

Let $C$ be a free chain complex over a  principal ideal domain, $R$, and suppose that $M$ is an $R$--module. Then there is a
functorial short exact sequence
$$0\to H_q(C)\otimes_R M\to H_q(C;M)\to \mathrm{Tor}_R(H_{q-1}(C),M)\to 0.$$
This exact sequence is split, but the splitting may not be
functorial.
\end{thm}

We wish to apply the universal coefficients theorem to understand the Heegaard Floer homology with $\F$ coefficients through an understanding of the Floer homology with twisted coefficients, where the twisted coefficient ring is $\mathcal R=\F[T,T^{-1}]$.  Viewing $\F$ as the trivial $\mathcal R$ module (where $T$ acts as $1$), the following lemma analyzes the tensor product in the universal coefficient splitting.

\begin{lem}\label{lem:Tor}
Suppose $M$ is a finitely generated module over $\mathcal R=\F[T,T^{-1}]$,
$M^{\mathrm{tors}}$ is the submodule of $M$  consisting of all
torsion elements, and $M^{\mathrm{free}}=M/M^{\mathrm{tors}}$.
Then there is a short exact sequence
$$0\to M^{\mathrm{tors}}\otimes_{\mathcal R}\mathbb
F\to M\otimes_{\mathcal R}\mathbb F\to
M^{\mathrm{free}}\otimes_{\mathcal R}\mathbb F\to 0.$$ Moreover,
$$M^{\mathrm{tors}}\otimes_{\mathcal R}\mathbb
F\cong\mathrm{Tor}_{\mathcal R}(M,\mathbb F).$$
\end{lem}
\begin{proof}
The short exact sequence
$$0\to M^{\mathrm{tors}}\to M\to M^{\mathrm{free}}\to 0$$
gives rise to a long exact sequence
$$\cdots\to\mathrm{Tor}_1^{\mathcal R}(M^{\mathrm{free}},\mathbb F)\to M^{\mathrm{tors}}\otimes_{\mathcal R}\mathbb
F\to M\otimes_{\mathcal R}\mathbb F\to
M^{\mathrm{free}}\otimes_{\mathcal R}\mathbb F\to 0.$$ Since
$M^{\mathrm{free}}$ is free, $\mathrm{Tor}_{\mathcal
R}(M^{\mathrm{free}},\mathbb F)=0$, hence we have the desired
short exact sequence.

Since $\mathcal R$ is a principal ideal domain $$M^{\mathrm{tors}}\cong\bigoplus_i \mathcal
R/(p_i^{k_i}),$$ where $p_i$'s are prime elements in $\mathcal R$,
$k_i\in\mathbb Z_{\ge1}$. Note that $\mathbb F\cong\mathcal
R/(T-1)$. If $p_i\ne(T-1)$ up to a unit, then
$$\mathcal
R/(p_i^{k_i})\otimes_{\mathcal R}\mathbb F=0,\quad
\mathrm{Tor}_{\mathcal R}(\mathcal R/(p_i^{k_i}),\mathbb F)=0.$$
If $p_i=(T-1)$ up to a unit, then
$$\mathcal
R/(p_i^{k_i})\otimes_{\mathcal R}\mathbb F\cong \F,\quad
\mathrm{Tor}_{\mathcal R}(\mathcal R/(p_i^{k_i}),\mathbb F)\cong
p_i^{k_i-1}\mathcal R/(p_i^{k_i})\cong \F.$$ Hence our result
follows.
\end{proof}

\bigskip
\noindent Theorem~\ref{thm:HFmodUnique}  states that if  $\widehat{HF}(Y;\F)\cong \F[X_1,...,X_{n-1}]/(X_1^2,..., X_{n-1}^2)$ as a module,   then $Y\cong M\#(\#^{n-1}(S^1\times S^2))$, where $M$ is an integer homology sphere satisfying $\HFa(M)\cong \F$.   We turn to  the proof of this theorem.
\bigskip

\noindent {\bf Proof of Theorem~\ref{thm:HFmodUnique}.}
We first reduce to the case that $Y$ is irreducible.  Suppose  that $Y$ is a nontrivial connected sum.  Then we can apply (\ref{eq:ConnSum}) to restrict our attention to a connected summand. If $Y=S^1\times S^2$, then our conclusion holds. So we may assume that $Y$ is irreducible. 

Let $$\Lambda_{n-1}=\mathbb
F [X_1,\dots,X_{n-1}]/(X_1^2,\dots,X_{n-1}^2).$$ Let
$\zeta_1,\dots,\zeta_{n-1}$ be elements in $H_1(Y;\mathbb
Z)/\mathrm{Tors}$ such that $\Curve^{\zeta_i}(\mathbf 1)=X_i$. If
$S=\{i_1,\dots,i_k\}\subset\{1,\dots,n\}$, let
$$\Curve^S=\Curve^{\zeta_{i_1}}\circ\cdots\circ \Curve^{\zeta_{i_k}}$$
and let $$X_S=X_{i_1}X_{i_2}\cdots X_{i_k}\in\widehat{HF}(Y;\mathbb F )\cong\Lambda_{n-1}.$$ It follows that $\Curve^S(\mathbf 1)=X_S$. Since
$\Curve^{[\zeta_i]}$ decreases the Maslov grading by $1$, we can give a
relative Maslov grading to $\widehat{HF}(Y;\mathbb F )$
such that the grading of $X_S$ is $n-1-|S|$.

By
Corollary~\ref{cor:ConnNonzero}, there exists 
$[\omega]\in H^2(Y;\mathbb Z)$ such that
$H(\underline{\widehat{HF}}(Y;\mathcal
R_{[\omega]}),\Curve^{[\zeta]})$ is non-torsion for any $[\zeta]\in
H_1(Y)$. Let
$C=\underline{\widehat{CF}}(Y;\mathcal R_{[\omega]})$. Let
$H^{\mathrm{tors}}(C)$ be the submodule of $H(C)$ which consists
of all torsion elements in $H(C)$, and let
$H^{\mathrm{free}}(C)=H(C)/H^{\mathrm{tors}}(C)$. By
Theorem~\ref{thm:UCT}, there is a short exact sequence
\begin{equation}\label{eq:ModExact}
\begin{CD}
0\to H_*(C)\otimes_{\mathcal R}\mathbb
F @>\mu_*>>H_*(C\otimes_{\mathcal R}\mathbb
F )@>\tau_*>>\mathrm{Tor}_{\mathcal R}(H_{*-1}(C),\mathbb
F )\to0.
\end{CD}
\end{equation}
Moreover, by the functoriality of (\ref{eq:ModExact}), the three
groups are $\Lambda_{n-1}$--modules such that the exact sequence
respects the module structure. Recall that $$H_*(C\otimes_{\mathcal R}\F)\cong\widehat{HF}(Y;\F)\cong\Lambda_{n-1}.$$

\noindent{\bf Claim 1.} The map $\mu_{n-1}$ is zero.

If the map $\mu_{n-1}$ is nonzero, then it is an isomorphism since
$H_{n-1}(C\otimes_{\mathcal R}\mathbb F )$ is one-dimensional.
Since $\mathbf 1\in H_{n-1}(C\otimes_{\mathcal R}\mathbb F )$
generates the module $H_{*}(C\otimes_{\mathcal R}\mathbb F )$,
the map $\mu_*$ is surjective. Hence $\mu_*$ is an isomorphism and
$\mathrm{Tor}_{\mathcal R}(H(C),\mathbb F )=0$.
Lemma~\ref{lem:Tor} implies that
$H^{\mathrm{tors}}(C)\otimes_{\mathcal R}\mathbb F =0$ and the
module $H(C)\otimes_{\mathcal R}\mathbb F $ is isomorphic to the
module $H^{\mathrm{free}}(C)\otimes_{\mathcal R}\mathbb F $. Now
we have
$$H(H^{\mathrm{free}}(C)\otimes_{\mathcal R}\mathbb
F ,\Curve^{\zeta_1})\cong H(H(C)\otimes_{\mathcal R}\mathbb
F ,\Curve^{\zeta_1})\cong H(H(C\otimes_{\mathcal R}\mathbb
F ),\Curve^{\zeta_1})\cong 0,$$ a contradiction to the fact that
$H(H(C),\Curve^{\zeta_1})$ has positive rank.

\noindent{\bf Claim 2.} The map $\tau_{1}$ is zero.

If $\tau_1$ is nonzero, then there exists an $(n-2)$--element
subset $S\subset\{1,\dots,n-1\}$ such that $\tau_1(X_S)\ne0$. We
claim that the $2^{n-2}$ elements
$$\Curve^{S'}\circ\tau(\mathbf 1),\quad S'\subset S$$ are linearly
independent over $\mathbb F $. In fact, suppose $S_1,\dots,S_m$
are subsets of $S$ with $|S_i|=k$, we want to show that
\begin{equation}\label{eq:LinInd}
\sum_i\Curve^{S_i}\circ\tau(\mathbf 1)\ne0.\end{equation} Apply
$\Curve^{S\backslash S_1}$ to the left hand side of (\ref{eq:LinInd}).
Since $(S\backslash S_1)\cap S_i\ne\emptyset$ for all $i\ne1$,
$\Curve^{S\backslash S_1}\Curve^{S_i}=0$ when $i\ne1$. So we get
\begin{eqnarray*}
\Curve^{S\backslash S_1}(\sum_i\Curve^{S_i}\circ\tau(\mathbf1))
&=&\Curve^{S\backslash S_1}\Curve^{S_1}\circ\tau(\mathbf 1)\\
&=&\Curve^S\circ\tau(\mathbf 1)\\
&=&\tau\circ \Curve^S(\mathbf 1)\\
&=&\tau(X_S)\\
&\ne&0.
\end{eqnarray*}
So (\ref{eq:LinInd}) holds.

Now we have proved that the rank of $\mathrm{Tor}_{\mathcal
R}(H(C),\mathbb F )$ is at least $2^{n-2}$, which is half of the
rank of $H_*(C\otimes_{\mathcal R}\mathbb F )$.  By (\ref{eq:ModExact}), we have $$ H_*(C\otimes_{\mathcal R}\mathbb F )\cong (H(C)\otimes_{\mathcal R}\F)\oplus\mathrm{Tor}_{\mathcal R}(H(C),\F).$$ Using Lemma~\ref{lem:Tor}, we see that
$H^{\mathrm{free}}(C)\otimes_{\mathcal R}\mathbb F =0$, which
contradicts the fact that $H(C)$ has positive rank. This finishes the
proof of Claim~2.

By Claim~2 we have $\mathrm{Tor}_{\mathcal R}(H_0(C),\mathbb
F )=0$. So
\begin{equation}\label{eq:TorsF=0}
H_0^{\mathrm{tors}}(C)\otimes_{\mathcal R}\mathbb F =0
\end{equation}
 by
Lemma~\ref{lem:Tor}. By Claim~1
$$H_{n-2}^{\mathrm{tors}}(C)\otimes_{\mathcal R}\mathbb
F \cong\mathrm{Tor}_{\mathcal R}(H_{n-2}(C),\mathbb F )\ne0.$$
Let $u\in H_{n-2}^{\mathrm{tors}}(C)\otimes_{\mathcal R}\mathbb
F $ be a nonzero element, then $\mu(u)\ne0$ and there exists an
$(n-2)$--element subset $S\subset\{1,\dots,n-1\}$ such that
$\Curve^S\circ\mu(u)$ is the generator of $H_0(C\otimes_{\mathcal
R}\mathbb F)$. Thus $\mu\circ \Curve^S(u)\ne0$.

Since $u\in H_{n-2}^{\mathrm{tors}}(C)\otimes_{\mathcal R}\mathbb
F $, $\Curve^S(u)\in H_0^{\mathrm{tors}}(C)\otimes_{\mathcal R}\mathbb F \cong0$ by (\ref{eq:TorsF=0}).
Thus $\mu\circ \Curve^S(u)=0$, a contradiction. \qed

\bigskip

\noindent With the detection theorem in hand, we can easily prove that the Khovanov module detects unlinks:
\bigskip

\noindent {\bf Proof of Theorem~\ref{thm:main}.}
It follows from the module structure of $Kh(L)$ that
$\Khr(L,L_0)\cong\Lambda_{n-1}=\F[X_1,\dots,X_{n-1}]/(X_1^2,\dots,X_{n-1}^2)$.
By Proposition \ref{collapse}, $\widehat{HF}(\Sigma(L))\cong\Lambda_{n-1}$ as a module.

By Theorem~\ref{thm:HFmodUnique}, $\Sigma(L)\cong M\#(\#^{n-1}(S^1\times S^2))$, where $M$ is an integral homology sphere with $\widehat{HF}(M)\cong\F$.
If a link $J$ is non-split, then $\Sigma(J)$ does not contain an $S^1\times
S^2$ connected summand; on the other hand, if $J=J_1\sqcup J_2$, then $\Sigma(J)=\Sigma(J_1)\#\Sigma(J_2)\#(S^1\times S^2)$ \cite[Proposition~5.1]{HN}. Applying this fact to the link $L$ at hand, it follows that $L=L_0\sqcup L_1\sqcup\cdots\sqcup L_{n-1}$. Since $Kh(J_1\sqcup J_2)\cong Kh(J_1)\otimes Kh(J_2)$, each $L_i$
has $\mathrm{rank}\:Kh(L_i)=2$. It follows from \cite{KM2010} that each $L_i$ is an unknot, so $L$ is an unlink.
\qed

\end{document}